\documentclass[11pt]{amsart}

\usepackage{epigamath}

\usepackage[english]{babel}

\numberwithin{equation}{section}

\usepackage[shortlabels]{enumitem}
\setlist[enumerate,1]{label={\rm(\arabic*)}, ref={\rm\arabic*}} 

\usepackage{amsfonts}
\usepackage{amsthm}
\usepackage{amsmath}
\usepackage{amscd}
\usepackage[latin2]{inputenc}
\usepackage{t1enc}
\usepackage[mathscr]{eucal}
\usepackage{indentfirst}
\usepackage{graphicx}
\usepackage{graphics}
\usepackage{pict2e}
\usepackage{pgfplots}
\usepackage{epic}
\usepackage{epstopdf} 
\usepackage{tikz-cd}
\usepackage{mathrsfs}

\newtheorem{theorem}{Theorem}[section]
\newtheorem{corollary}[theorem]{Corollary}
\newtheorem{lemma}[theorem]{Lemma}
\newtheorem{proposition}[theorem]{Proposition}
\newtheorem{defthm}[theorem]{Definition-Theorem}

\theoremstyle{definition}
\newtheorem{definition}[theorem]{Definition}

\theoremstyle{remark}
\newtheorem{example}[theorem]{Example}
\newtheorem{remark}[theorem]{Remark}

\DeclareMathOperator{\im}{im}

\DeclareMathOperator{\Hom}{Hom}

\DeclareMathOperator{\Def}{Def}
\DeclareMathOperator{\codim}{codim}
\DeclareMathOperator{\Mon}{Mon}
\DeclareMathOperator{\id}{id}
\DeclareMathOperator{\Sym}{Sym}
\DeclareMathOperator{\Tr}{Tr}
\DeclareMathOperator{\gr}{gr}
\DeclareMathOperator{\Supp}{Supp}
\DeclareMathOperator{\Aut}{Aut}
\DeclareMathOperator{\SO}{SO}
\DeclareMathOperator{\IH}{IH}
\DeclareMathOperator{\SH}{SH}
\DeclareMathOperator{\SIH}{SIH}

\newcommand{\lt}{\mathrm{lt}}

\makeatletter
\newcommand{\extp}{\@ifnextchar^\@extp{\@extp^{\,}}}
\def\@extp^#1{\mathop{\bigwedge\nolimits^{\!#1}}}
\makeatother

\makeatletter
\newcommand*\bigcdot{\mathpalette\bigcdot@{.5}}
\newcommand*\bigcdot@[2]{\mathbin{\vcenter{\hbox{\scalebox{#2}{$\m@th#1\bullet$}}}}}
\makeatother

\def\lowsim{\vbox to 0pt{\vss\hbox{$\sim$}\vskip-3pt}}
\newcommand{\lra}{\longrightarrow}
\newcommand{\longhookrightarrow}{\lhook\joinrel\longrightarrow}
\newcommand{\supst}[1]{\ensuremath{#1^{\mathrm{st}}}}
\newcommand{\supth}[1]{\ensuremath{#1^{\mathrm{th}}}}

\usepackage{url}

\EpigaVolumeYear{9}{2025} \EpigaArticleNr{6} \ReceivedOn{August 25, 2023}
\InFinalFormOn{July 9, 2024}
\AcceptedOn{August 17, 2024}

\title{The Looijenga--Lunts--Verbitsky Algebra for Primitive Symplectic Varieties with Isolated Singularities}
\titlemark{The LLV Algebra for Primitive Symplectic Varieties}

\author{Benjamin Tighe}
\address{Department of Mathematics, University of Oregon, Eugene, OR 97403, USA}
\email{bentighe@uoregon.edu}

\authormark{B.~Tighe}

\AbstractInEnglish{We extend results of Looijenga--Lunts and Verbitsky and show that the total Lie algebra $\mathfrak g$ for the intersection cohomology of a primitive symplectic variety $X$ with isolated singularities is isomorphic to
  $$
  \mathfrak g \cong \mathfrak{so}\left(\left(\IH^2(X, \mathbb Q), Q_X\right)\oplus \mathfrak h\right), 
  $$
  where $Q_X$ is the intersection Beauville--Bogomolov--Fujiki form and $\mathfrak h$ is a hyperbolic plane.  This gives a new, \textit{algebraic} proof for irreducible holomorphic symplectic manifolds which does not rely on the hyperk\"ahler metric.

  Along the way, we study the structure of $\IH^*(X, \mathbb Q)$ as a $\mathfrak{g}$-representation---with particular emphasis on the Verbitsky component, multidimensional Kuga--Satake constructions, and Mumford--Tate algebras---and give some immediate applications concerning the $P = W$ conjecture for primitive symplectic varieties.}

\MSCclass{32S35, 14C30, 32S60, 32S50 (primary); 14D06, 14B05, 14E15, 32S20 (secondary)}

\KeyWords{Hyperk\"ahler manifolds, symplectic varieties, intersection cohomology, semismall morphisms}

\begin{document}



\maketitle

\begin{prelims}

\DisplayAbstractInEnglish

\bigskip

\DisplayKeyWords

\medskip

\DisplayMSCclass

\end{prelims}


\newpage

\setcounter{tocdepth}{1}

\tableofcontents


\section{Introduction}

\subsection{Background} Hyperk\"ahler manifolds are distinguished in complex algebraic geometry due to their rich Hodge theory and form one of the three building blocks of the Beauville--Bogomolov decomposition of K-trivial varieties.  Verbitsky's global Torelli theorem, see \cite{verbitsky2009global}, states that hyperk\"ahler manifolds are essentially determined by their second cohomology, along with its monodromy representation.  One then expects the higher cohomology groups to be determined by the Hodge theory of $H^2$.  This can be described using the Looijenga--Lunts--Verbitsky (LLV) algebra, a Lie algebra on the total cohomology $H^*(X, \mathbb Q)$ of a compact hyperk\"ahler manifold~$X$, which was studied independently by Looijenga--Lunts \cite{looijenga1997lie} and by Verbitsky \cite{verbitsky1996cohomology} in his thesis.
    
We say a class $\omega \in H^2(X, \mathbb Q)$ is HL if it satisfies the hard Lefschetz theorem: For every $k$, the cupping morphism $L_\omega$ on cohomology gives isomorphisms
$$
L_\omega^k\colon H^{\dim X-k}(X, \mathbb Q) \overset{\lowsim}\lra H^{\dim X + k}(X, \mathbb Q).
$$
Equivalently, $\omega$ is HL if the nilpotent operator $L_\omega$ completes to an $\mathfrak{sl}_2$-triple $\mathfrak g_\omega = \langle L_\omega, \Lambda_\omega, H\rangle$, where $H = [L_\omega, \Lambda_\omega]$ acts as $(k-\dim X)\id$ on $H^k(X, \mathbb Q)$.  We define the \textit{LLV algebra} of a hyperk\"ahler manifold $X$ to be the Lie algebra generated by all possible  $\mathfrak{sl}_2$-algebras of the cohomology ring $H^*(X)$ coming from hard Lefschetz operators: $$
\mathfrak{g} = \langle L_\omega, \Lambda_\omega:\omega~\mathrm{is~HL}\rangle.
$$
Looijenga--Lunts and Verbitsky show that the LLV algebra admits a natural isomorphism 
    \begin{equation} \label{LLVstructurehyperkahler} 
    \mathfrak g \cong \mathfrak{so}\left(\left(H^2(X, \mathbb Q),q_X\right)\oplus \mathfrak h\right), 
    \end{equation}
    where $q_X$ is the Beauville--Bogomolov--Fujiki form and $\mathfrak h$ a hyperbolic plane.  The LLV algebra not only acts on $H^*(X, \mathbb Q)$, but the structure theorem shows that the algebra is dependent only on the pair $(H^2(X, \mathbb Q),q_X)$ and is therefore a deformation invariant.  In fact, the various Hodge structures of the higher cohomology groups corresponding to deformations of $X$ are detected by the representation theory of $\mathfrak g$, as $\mathfrak g_{\mathbb R} : = \mathfrak{so}((H^2(X, \mathbb R), q_X)\oplus \mathfrak h)$ contains the Weil operators $C = i(p-q)\id$ for all complex structures on $X$.
    
    In recent years, there has been progress in generalizing the Hodge theory of hyperk\"ahler manifolds to the singular setting.  A \textit{primitive symplectic variety} is a normal compact K\"ahler variety $X$ such that $H^1(\mathscr O_X) = 0$ and the regular locus $U$ admits a global holomorphic symplectic form $\sigma$ which extends holomorphically across any resolution of singularities and satisfies $H^0(U, \Omega_U^2) = \mathbb C\cdot \sigma$.  Such varieties also enjoy a rich Hodge theory: By work of Bakker--Lehn \cite{bakker2018global}, the second cohomology group of a primitive symplectic variety~$X$ carries a pure Hodge structure and admits a version of global Torelli, which for \textit{$\mathbb Q$-factorial terminal singularities} says that $X$ is essentially recovered by its $H^2$.  It is then natural  to ask if there is a generalization of the LLV algebra for primitive symplectic varieties which encodes the Hodge theory.
    \subsection{Main Results}  
        
        \subsubsection{The LLV algebra for intersection cohomology}    
        
        Constructing the LLV algebra for the \textit{ordinary} cohomology of a primitive symplectic variety is difficult since $H^*(X, \mathbb Q)$, \textit{a priori}, neither carries a pure Hodge structure nor satisfies the hard Lefschetz theorem.  Instead, we work with the \textit{intersection cohomology groups}.
        
        Intersection cohomology was invented by Goresky--MacPherson \cite{goresky1980intersection} as a way of generalizing Poincar\'e duality to singular topological spaces.  Beilinson--Bernstein--Deligne \cite{deligne1983faisceaux} observed, using characteristic $p$ methods, that the intersection cohomology groups of a projective variety admit a decomposition theorem with respect to projective morphisms.  As a consequence, the intersection cohomology groups carry pure Hodge structures and satisfy the hard Lefschetz theorem.  This was also observed by Saito \cite{saito1988modules} in greater generality using the theory of mixed Hodge modules, as well as work of de Cataldo--Migliorini \cite{de2005hodge} using purely Hodge theoretic techniques.  The goal of this paper is to understand the total Lie algebra with respect to intersection cohomology, which we define analogously as the Lie algebra generated by the $\mathfrak{sl}_2$-operators corresponding to any HL class, \textit{i.e.}, those classes $\omega \in \IH^2(X, \mathbb Q)$ such that
        $$
        L_\omega^k\colon \IH^{\dim X - k}(X, \mathbb Q) \overset{\lowsim}\lra \IH^{\dim X + k}(X, \mathbb Q).
        $$
        
        To this end, we define a Beauville--Bogomolov--Fujiki (BBF) form $Q_X$ on the intersection cohomology $\IH^2(X, \mathbb Q)$ of a primitive symplectic variety $X$; see Section~\ref{subsubsection bbf form}.  It is compatible with the standard BBF form $q_X$ on $H^2(X, \mathbb Q)$ (see Definition~\ref{BBFdef}) and satisfies
        $$
        Q_X|_{H^2(X, \mathbb Q)} = q_X
        $$
        corresponding to the natural inclusion $H^2(X, \mathbb Q) \subset \IH^2(X, \mathbb Q)$ (see Remark~\ref{remark intersection cohomology inclusion rational sings}).
        
        \begin{theorem} \label{LLVint}
          Let $X$ be a primitive symplectic variety with isolated singularities and $b_2 \ge 5$ and $\mathfrak g$ the algebra generated by all $\mathfrak{sl}_2$-triples corresponding to HL classes in $\IH^2(X, \mathbb Q)$.  There are isomorphisms
          $$
          \mathfrak g \cong \mathfrak{so}\left(\left(\IH^2(X, \mathbb Q), Q_X\right)\oplus \mathfrak h\right), \quad \mathfrak g_{\mathbb R} \cong \mathfrak{so}(B_2-2,4), 
          $$
          where $B_2 = \dim \IH^2(X, \mathbb Q)$.
        
        Moreover, a Hodge structure on $\IH^*(X, \mathbb Q)$ is determined by a Hodge structure on $\IH^2(X, \mathbb Q)$ and the action of\, $\mathfrak g$ on $\IH^*(X, \mathbb Q)$. 
         \end{theorem}

         The assumption on $b_2$ is due to our use of the global moduli theory of Bakker--Lehn \cite{bakker2018global}.  We note that the case $b_2 \le 4$ holds assuming the surjectivity of the period map and other special cases (see Section~\ref{subsection remark on b_2}).  We emphasize the fact that our proof of Theorem~\ref{LLVint} gives an \textit{algebraic proof} of (\ref{LLVstructurehyperkahler}).  We expect our methods to generalize to any primitive symplectic variety.
        
         \subsubsection{Symplectic symmetry on the intersection cohomology groups}  One of the key features of the cohomology of a hyperk\"ahler manifold $X$ is its structure as an \textit{irreducible holomorphic symplectic manifold}.  The holomorphic symplectic form $\sigma$ on $X$ induces isomorphisms $\Omega_X^{n-p} \xrightarrow{\lowsim} \Omega_X^{n+ p}$ by wedging, where $2n$ is the (complex) dimension of $X$.  Passing to cohomology, we get the \textit{symplectic hard Lefschetz theorem}
         $$
         L_\sigma^p\colon H^{n-p,q}(X) \overset{\lowsim}\lra H^{n+p,q}(X), 
         $$
         which induces the extra symmetry on the Hodge diamond of $X$. An interesting observation is that, by deforming a compact hyperk\"ahler manifold $X$, we can see that the hard Lefschetz theory of $H^*(X, \mathbb Q)$ is related to its symplectic hard Lefschetz theory due to Verbitsky's global Torelli theorem.  With this in mind, we first show that the intersection cohomology groups of primitive symplectic varieties with isolated singularities also admit this symplectic symmetry.
        
        \begin{theorem} \label{symHL}
          Let $X$ be a primitive symplectic variety of dimension $2n$ with isolated singularities, and let $\IH^{p,q}(X) \subset \IH^k(X, \mathbb C)$ be the $(p,q)$-part of the canonical Hodge structure on $\IH^k(X, \mathbb Q)$.  There is a cupping morphism $L_\sigma\colon \IH^{p,q}(X) \to \IH^{p+2,q}(X)$ on the Hodge pieces of the intersection cohomology which induced isomorphisms
          $$
          L_\sigma^p\colon \IH^{n-p,q}(X) \overset{\lowsim}\lra \IH^{n+p,q}(X).
          $$
        \end{theorem} 
        To prove Theorem~\ref{symHL}, we study the Hodge theory of the (compactly supported) cohomology of the regular locus $U : = X_{\mathrm{reg}}$.  We prove the following useful theorem, giving an analog of \cite[Theorem~2]{arapura1990local}
        
        \begin{theorem}\label{Hodge theory regular locus}
          Let $X$ be a primitive symplectic variety of dimension $2n$ with regular locus $U$.  Suppose that the singular locus of\, $X$ is smooth. The Hodge-to-de Rham spectral sequence
          $$
          E_1^{p,q} = H^q\left(U, \Omega_U^p\right) \Longrightarrow H^{p + q}(U, \mathbb C)
          $$
          degenerates at $E_1$ for $p + q < 2n-1$.
        \end{theorem}
        
        The LLV structure theorem follows from the symplectic hard Lefschetz theory.  Theorem~\ref{symHL} shows that there are operators $L_\sigma, \Lambda_\sigma$ which complete to an $\mathfrak{sl}_2$-triple
        $$
        \mathfrak s_{\sigma} = \langle L_\sigma, \Lambda_\sigma, H_\sigma\rangle, 
        $$
        where $H_\sigma$ acts as the holomorphic weight operator $H_\sigma(\alpha) = (p-n)\alpha$ for an intersection $(p,q)$-class.  Similarly, by conjugation we get a second $\mathfrak{sl}_2$-triple
        $$
        s_{\overline \sigma} = \langle L_{\overline \sigma}, \Lambda_{\overline \sigma}, H_{\overline \sigma}\rangle
        $$
        corresponding to the antiholomorphic symplectic form $\overline \sigma$, where $H_{\overline \sigma}(\alpha) = (q-n)\alpha$.  This generates an $\mathfrak{sl}_2\times \mathfrak{sl}_2$-structure on the total intersection cohomology  $\IH^*(X)$.  The key observation is that the Lefschetz operators for $\sigma$ and $\overline \sigma$ commute:
        $$
        [L_\sigma, L_{\overline{\sigma}}] = [\Lambda_\sigma, \Lambda_{\overline \sigma}] = 0.
        $$
        The representation theory of this $\mathfrak{sl}_2\times \mathfrak{sl}_2$-action, along with the monodromy representation of $H^2(X, \mathbb C)$, describes the LLV algebra $\mathfrak g$ of intersection cohomology completely.  This will lead to the proof of Theorem~\ref{LLVint}.

        \begin{remark}
         We expect our methods to generalize to any primitive symplectic variety, although this will require a better understanding of the Hodge theory of the intersection cohomology groups.  One case where our methods generalize is the case of symplectic orbifolds (see Proposition~\ref{proposition LLV for orbifolds}), although the LLV structure theorem should be known to experts due to the existence of hyperk\"ahler metrics.
        \end{remark}
        
    \subsection{Representation Theory and Hodge Theory of the LLV Algebra} The structure of the cohomology ring $H^*(X, \mathbb Q)$ of a compact hyperk\"ahler manifold $X$ as a $\mathfrak g$-representation has been studied in recent years, leading to many interesting results and conjectures concerning these varieties.  We extend some of these results to the intersection cohomology module of a primitive symplectic variety with isolated singularities.  
    
    \subsubsection{LLV decomposition} Extending a result of Verbitsky \cite{verbitsky1996cohomology}, we show that the module generated by $\IH^2(X, \mathbb C)$ in $\IH^*(X, \mathbb C)$ is a $\mathfrak g$-module, called the \textit{Verbitsky component} $V_{(n)}$.  In \cite{green2019llv}, the Verbitsky component and the LLV decomposition were studied for the known examples of compact hyperk\"ahler manifolds.  In general, it is expected that $V_{(n)}$ (along with the LLV decomposition) puts restrictive conditions on the cohomology of a compact hyperk\"ahler manifold, and we expect the same to be true for primitive symplectic varieties.
    
    \subsubsection{Kuga--Satake construction} Recall that the classical Kuga--Satake construction, see \cite{kuga1967abelian}, associates to a K3 surface $S$ an abelian variety $A$ and an embedding $H^2(S) \hookrightarrow H^1(A)\otimes H^1(A)^*$ of polarized weight~2 Hodge structures (see for example \cite[Section~2.6]{huybrechts2016lectures}).  In \cite{kurnosov2019kuga}, this construction was extended to compact hyperk\"ahler manifolds and the LLV algebra: If $X$ is compact hyperk\"ahler, there exist an abelian variety $A$ and an embedding $\mathfrak g_X \hookrightarrow \mathfrak g_A$, where we note that $\mathfrak g_A \cong \mathfrak{so}(H^1(A)\otimes H^1(A)^*)$ by \cite[Proposition~(3.3)]{looijenga1997lie}.   We outline how the same construction associates to a primitive symplectic variety with isolated singularities a complex torus and an embedding of the total Lie algebras.
    
    \subsubsection{Mumford--Tate algebras}  The Mumford--Tate algebra of a pure Hodge structure $H$ is the smallest $\mathbb Q$-algebraic subalgebra $\mathfrak m$ of $\mathfrak{gl}(H)$ for which $\mathfrak m_{\mathbb R}$ contains the Weil operator.  There is a relationship between the Mumford--Tate algebra and the LLV algebra, which was studied thoroughly in \cite{green2019llv}.  We make similar observations in the isolated singularities case and show that the Mumford--Tate algebra sits naturally inside the LLV algebra for intersection cohomology.  Moreover, we show that the degree of transcendence of the total intersection cohomology $\IH^*(X, \mathbb C)$ over the special Mumford--Tate group is equal to that of $\IH^2(X, \mathbb C)$, which is the second statement of Theorem~\ref{LLVint}.
         
    \subsection{$\boldsymbol{P = W}$ for Primitive Symplectic Varieties} The $P = W$ conjecture for compact hyperk\"ahler manifolds asserts that the perverse filtration on the cohomology $H^*(X, \mathbb C)$ induced by a Lagrangian fibration agrees with the weight filtration of the limit mixed Hodge structure of a type III degeneration\footnote{Recall that a degeneration $\mathscr X \to \Delta$ is of type III if the degeneration has maximally unipotent monodromy operator $T$ and the nilpotent log-monodromy operator $N$ is of index 3.} of $X$, which was shown to exist in \cite{soldatenkov2018limit}.  In \cite{harder2021p}, the two filtrations were shown to agree by showing their corresponding weight operators define the same element in the LLV algebra. 
         
    More generally, let $X$ be a primitive symplectic variety.  We show that there is a good notion of degeneration for primitive symplectic varieties via locally trivial deformations, which respects Schmid's nilpotent orbit theory; see \cite{schmid1973variation}. We also show that, corresponding to a degeneration of primitive symplectic varieties, there is a \textit{limit mixed Hodge structure} on the intersection cohomology of the central fiber (see Section~\ref{subsubsection lmhs}).  Following \cite{soldatenkov2018limit, harder2021p}, we show the following. 
         
         \begin{theorem}
         Let $X$ be a primitive symplectic variety with isolated singularities and $b_2 \ge 5$.   
        \begin{enumerate}
            \item There exists a type III degeneration $\mathscr X \to \Delta$ of\, $X$ whose logarithmic monodromy operator $N$ has index 3.
            
            \item If\, $X$ admits the structure of a Lagrangian fibration $f:X \to B$, then the perverse filtration $P_\beta$ on $\IH^*(X, \mathbb C)$ associated to the pullback of an ample class on $B$ agrees with the weight filtration of the limit mixed Hodge structure of the degeneration $\mathscr X \to \Delta$.
        \end{enumerate}
         \end{theorem}

    \subsection{Outline}
    
    The paper is organized as follows.  In Section~\ref{2}, we review some results of primitive symplectic varieties which will be used throughout this paper, as well as the relevant properties of intersection cohomology and mixed Hodge structures.  We prove two auxiliary results that will allow us to simplify our assumptions: The first says that bimeromorphic morphisms of primitive symplectic varieties are semismall, a generalization of Kaledin's result for symplectic resolutions \cite{kaledin2006symplectic}.  The second, which is most likely known to experts, is a criterion for $\mathbb Q$-factoriality for a primitive symplectic variety $X$ in terms of the inclusion $H^2(X, \mathbb Q) \hookrightarrow \IH^2(X, \mathbb Q)$; see Proposition~\ref{Qcrit}.

In Section~\ref{3}, we prove that the intersection cohomology groups satisfy a symplectic hard Lefschetz theorem.  We do this by studying the extension of differential forms across singularities and the Hodge-to-de Rham spectral sequence on the regular locus.

In Section~\ref{4}, we use the symplectic symmetry from Section~\ref{3} to construct dual Lefschetz operators with respect to the symplectic forms $\sigma, \overline\sigma$ and study their commutator relations.  We then define non-isotropic classes $(\gamma,\gamma')$ satisfying $q_X(\gamma,\gamma') = 0$ and $[\Lambda_\gamma, \Lambda_{\gamma '}] = 0$, a key component in the structure of the LLV algebra.

In Section~\ref{5}, we show that the intersection cohomology groups satisfy the LLV algebra structure theorem, which is a consequence of the previous sections as well as the monodromy density theorem of Bakker--Lehn \cite[Theorem 1.1]{bakker2018global}.

In Section~\ref{7}, we discuss some representation-theoretic aspects of the LLV algebra.  We construct the Verbitsky component generated by $\IH^2$, extend the Kuga--Satake construction to intersection cohomology, and study the Mumford--Tate algebra in this setting.

In Section~\ref{8}, we describe a singular version of the $P = W$ theorem.

\subsection{Notation}  We work in the complex analytic category, and all varieties should be considered as complex analytic varieties unless otherwise stated.  

If $E \subset \widetilde{X}$ is a simple normal crossing divisor of a smooth complex manifold $\widetilde{X}$, we set
$$
\Omega_{\widetilde{X}}^p(\log E)(-E) : = \Omega_{\widetilde{X}}^p(\log E)\otimes_{\mathscr O_{\widetilde{X}}} \mathscr I_E,
$$
 where $\Omega_{\widetilde{X}}^p(\log E)$ is the sheaf of logarithmic $p$-forms and $\mathscr I_E$ is the ideal sheaf of $E$.

When we speak of intersection cohomology, we always mean with respect to the middle perversity.

Finally, if $(X, \sigma)$ is a primitive symplectic variety, we will think of $\sigma$ as a holomorphic form on the regular locus or a class in (intersection) cohomology without distinction.

\subsection*{Acknowledgements} This work is part of the author's Ph.D.\ thesis at the University of Illinois at Chicago.  I want to thank my advisor Benjamin Bakker for his insight into singular symplectic varieties, as well as the countless meetings devoted to this project.  I also want to thank Christian Lehn for comments on a preliminary draft of this paper. 

\section{Preliminaries and Auxiliary Results}\label{2} The main objects of study in this work are primitive symplectic varieties, and so we recall both the local and the global properties which will be used further on. 
We also review the basic properties of intersection cohomology and study their Hodge theory for primitive symplectic varieties. Finally, we give some auxiliary results regarding bimeromorphic morphisms of primitive symplectic varieties, which will be useful for certain reductions later in the paper.

    \subsection{Symplectic Varieties}\label{2.1}

    \subsubsection{MMP singularities}\label{2.1.1} Let $X$ be a normal complex variety.  A \textit{log-resolution of singularities} is a projective birational morphism $\pi\colon\widetilde{X} \to X$ from a smooth complex variety $\widetilde{X}$ which is an isomorphism over the regular locus $U$ and for which $\pi^{-1}(\Sigma) = E = \sum E_i$ is a simple normal crossing divisor, where $\Sigma$ is the singular locus of $X$ and the $E_i$ are the smooth components of $E$.

    We say $X$ is \textit{$\mathbb Q$-Gorenstein} if there is an integer $m$ such that $mK_X$ is Cartier, where $K_X$ is the canonical divisor; the smallest such integer is called the \textit{index} of $X$.  If $X$ has index 1, we say that $X$ is \textit{Gorenstein}.

    If $X$ is $\mathbb Q$-Gorenstein of index $m$ and $\pi\colon\widetilde{X} \to X$ a log-resolution of singularities, there are integers $a_i$ such that 
$$
mK_{\widetilde{X}} = \pi^*(mK_X) + \sum a_iE_i.
$$
  We say that $X$ has \textit{canonical} (resp.\ \textit{terminal}\,) singularities if we have $a_i \ge 0$ (resp.\ $a_i > 0$) for every $i$.  We call the $a_i$ the discrepancies of the exceptional divisors $E_i$.

    If $X$ is just a normal variety, we say that $X$ has \textit{rational singularities} if for some resolution of singularities $\pi\colon\widetilde{X} \to X$, the higher direct image sheaves satisfy $R^i\pi_*\mathscr O_{\widetilde{X}} = 0$ for every $i > 0$.

    Canonical singularities are rational; see \cite[Theorem 5.22]{kollar2008birational}.  Conversely, a normal variety with Gorenstein rational singularities has at worst canonical singularities.  One way to see this is to consider the \textit{holomorphic extension problem for differentials}: For which $p$ is the inclusion
     \begin{equation}\label{hol extension problem}
         \pi_*\Omega_{\widetilde{X}}^p \longhookrightarrow j_*\Omega_U^p
     \end{equation} is an isomorphism, where $j\colon U \hookrightarrow X$ is the inclusion of the regular locus $U : = X_{\mathrm{reg}}$?  Kebekus--Schnell showed if $X$ has rational singularities, then $\pi_*\Omega_{\widetilde{X}}^p \hookrightarrow j_*\Omega_U^p$ is an isomorphism for each $p$, see \cite[Corollary 1.8]{kebekus2021extending}, using the fact that the canonical sheaf $\omega_X$ is reflexive by Kempf's criterion.  In particular, the discrepancies $a_i$ are non-negative, as holomorphic $n$-forms on $U$ extend with at worst zeros.

    Since all our varieties are assumed normal, the sheaf $j_*\Omega_U^p$ is reflexive and isomorphic to the sheaf of \textit{reflexive $p$-forms} 
$$
\Omega_X^{[p]}: = (\Omega_X^p)^{**} \cong j_*\Omega_U^p.
$$
 Here $\Omega_X^1$ is the sheaf of K\"ahler differentials and $\Omega_X^p$ is the $\supth{p}$ exterior power.

    One application of the work of Kebekus--Schnell is the existence of a functorial pullback morphism for reflexive differentials; see \cite[Theorem 1.11]{kebekus2021extending}.  Given a morphism $f\colon Y \to X$ of reduced complex spaces with rational singularities, there is a pullback
    \begin{equation} \label{reflexive pullback}
        df\colon f^*\Omega_X^{[p]} \lra \Omega_Y^{[p]}
    \end{equation} which satisfies natural universal properties; see \cite[Section~14]{kebekus2021extending}.  Specifically, it agrees with the pullback of K\"ahler differentials on smooth varieties, whenever this makes sense.
        
        \subsubsection{Symplectic singularities}\label{2.1.2}

    \begin{definition}\label{sympdef}
    Let $X$ be a normal variety.  We say that $X$ is a \textit{symplectic variety} if there is a holomorphic symplectic form $\sigma \in H^0(U, \Omega_U^2)$ on the regular locus $U$ which extends to a holomorphic 2-form $\tilde \sigma \in H^0(\widetilde{X}, \Omega_{\widetilde{X}}^2)$ for any resolution of singularities $\pi\colon \widetilde{X} \to X$.
    \end{definition}

    Originally studied by Beauville \cite{beauville2000symplectic}, symplectic varieties were defined as an attempt to extend results on $K$-trivial manifolds to the singular setting.  Symplectic varieties have well-behaved singularities: They are rational Gorenstein, see \cite[Proposition 1.3]{beauville2000symplectic}, and therefore have at worst canonical singularities.  In fact, the strictly canonical locus is contained entirely in codimension 2 by \cite{namikawa2001note}.  Said differently, we have the following. 

    \begin{proposition}[\textit{cf.}~\protect{\cite[Corollary 1]{namikawa2001note}, \cite[Theorem 3.4]{bakker2018global}}]\label{symterm}
    A symplectic variety $X$ has terminal singularities if and only if\, $\codim_X(\Sigma) \ge 4$, where $\Sigma$ is the singular locus of\, $X$.
    \end{proposition}

    One of the key features of symplectic varieties is that they are stratified by varieties which, up to normalization, are again symplectic varieties.  We will use the following structure theorem throughout the paper. 

    \begin{proposition}[\textit{cf.} \protect{{\cite[Theorem 2.3,2.4]{kaledin2006symplectic}, \cite[Theorem 3.4]{bakker2018global}}}] \label{symstrat} Let $X$ be a symplectic variety.
    
    \begin{enumerate}
    \item There is a stratification  
$$
      X = X_0 \supset X_1 \supset X_2\supset \cdots
$$
      of\, $X$ given by the singular locus of\, $X$, so that  $X_i = (X_{i-1})_{\mathrm{sing}}$ for each $i$.  The normalization of each $X_{i}$ is a symplectic variety, and $X_i^\circ : = (X_i)_{\mathrm{reg}}$ admits a global holomorphic symplectic form.
    \item Suppose that $x \in X$ is such that $x \in X_{i}^\circ$.  Let $\widehat{X}_x$ and $\widehat{X_{i}^\circ}_x$ be the completions of\, $X$ and $X_{i}^\circ$ at $x$, respectively.  Then there is a decomposition
      $$
      \widehat{X}_x \cong Y_x\times \widehat{X_{i}^\circ}_x,
      $$
      where $Y_x$ is a symplectic variety.
    \end{enumerate} 
    \end{proposition} 

    In \cite{kaledin2009geometry}, a symplectic variety $Y_x$ is a formal scheme rather than the completion of some symplectic variety, as its existence is predicted by solutions to differential equations derived from the Poisson structure (see Section~3 of \textit{op. cit.}). 
    By considering \cite[Proposition 2.3, Appendix A]{kaplan2023crepant}, we can assume $Y_x$ is defined (and symplectic) in an analytic neighborhood of $x$.
    
    It would be interesting to understand how the holomorphic symplectic geometry of the regular strata $X_i^\circ$ determines the geometry of a symplectic variety $X$.  For instance, if $\sigma \in H^0(X, \Omega_X^{[2]})$ is the class of the symplectic form on the regular locus of $X$, there is class $j_i^*\sigma \in H^0(X_i^{\mathrm{nm}}, \Omega_{X_i}^{[2]})$, where $j_i\colon X_i^{\mathrm{nm}} \to X$ is the natural map from the normalization of $X_i$, defined by reflexive pullback for rational singularities; see \cite[Theorem 14.1]{kebekus2021extending}.  We claim this class is non-zero whenever $\dim X_i > 0$. Since the problem is local, we can consider the product decomposition $\widehat X_x \cong Y_x \times \widehat{X_i^\circ}_x$ by Proposition~\ref{symstrat}.  If $j_i^*\sigma = 0$, then $\sigma = p_1^*\sigma_{Y_x}$, where $p_1\colon X \to Y_x$ is the projection morphism and $\sigma_{Y_x}$ is the symplectic form on $Y_x$.  If $\dim Y_x \ne \dim X$, then $\sigma^{\dim X} : = \wedge^{\dim X}\sigma_{Y_x} = 0$.  This is absurd if $X$ is a symplectic variety, and so we see that $j_i^*\sigma$ defines a non-zero global section of $\Omega_{X_i}^{[2]}$.  Since $X_i^\circ$ is symplectic, $j_i^*\sigma$ is a symplectic form.

\subsubsection{Primitive symplectic varieties}\label{2.1.3}

We now transition to global properties of symplectic varieties.  The Hodge theory of singular symplectic varieties has been studied in  \cite{namikawa2000deformation,namikawa2005deformations,matsushita2001fujiki,matsushita2015base,schwald2017fujiki,bakker2016global,bakker2018global} at varying levels of generality; primitive symplectic varieties, which were studied in \cite{bakker2018global}, is the most general framework for studying the global properties of symplectic singularities.  

The general framework of the global moduli theory of primitive symplectic varieties works in the category of complex K\"ahler varieties.  A \textit{K\"ahler form} on a reduced complex analytic space $X$ is given by an open covering $X = \bigcup_i U_i$ and smooth strictly plurisubharmonic functions $f_i\colon U_i \to \mathbb R$ such that on each intersection $U_{ij}$, the function $f_{ij} = f_i|_{U_{ij}} - f_j|_{U_{ij}}$ is locally the real part of a harmonic function.  If $X$ admits a K\"ahler form, we say that $X$ is a \textit{K\"ahler variety}.  The most important property for this paper is that if $X$ is a K\"ahler variety, then $X$ admits a resolution by a K\"ahler manifold.  If $X$ is a compact K\"ahler variety with at worst rational singularities and $\pi\colon \widetilde{X} \to X$ is a resolution of singularities, then there is an injection $H^k(X, \mathbb Z) \hookrightarrow H^k(\widetilde{X}, \mathbb Z)$ for $k \le 2$, see \cite[Lemma 2.1]{bakker2016global}, and $H^k(X, \mathbb Z)$ inherits a pure Hodge structure from $H^k(\widetilde{X}, \mathbb Z)$, which is described as follows.  There is a spectral sequence 
$$
E_1^{p,q} : = H^q(X, j_*\Omega_U^p) \Longrightarrow \mathbb H^{p+q}(X, j_*\Omega_U^\bullet)
$$
 which degenerates at $E_1$ for $p + q \le 2$; see \cite[Lemma 2.2]{bakker2016global}.  For $k \le 2$, we have isomorphisms
\begin{equation}\label{reflexiveHTDR}
   H^k(X, \mathbb C) \overset{\lowsim}\lra \mathbb H^k(X, \pi_*\Omega_{\widetilde{X}}^\bullet) \cong \mathbb H^k(X, j_*\Omega_U^\bullet),
\end{equation} and the Hodge filtration is obtained by these isomorphisms and the degeneration of the reflexive Hodge-to-de Rham spectral sequence.

\begin{definition}\label{primsym}
A \textit{primitive symplectic variety} is a compact K\"ahler symplectic variety $(X, \sigma)$ such that $H^1(\mathscr O_X) = 0$ and $H^0(X, \Omega_X^{[2]}) = \mathbb C\cdot \sigma$.
\end{definition}

The second cohomology of a primitive symplectic variety is therefore a pure Hodge structure.  As in the case of irreducible holomorphic symplectic manifolds, the geometry is controlled by the Hodge theory on~$H^2$.  To see this, we need the following quadratic form. 

\begin{definition}\label{BBFdef}
Let $X$ be a primitive symplectic variety of dimension $2n$ with symplectic form $\sigma$.  We define a quadratic form $q_{X,\sigma}\colon H^2(X, \mathbb C) \to \mathbb C$ given by the formula 
$$
q_{X, \sigma}(\alpha) : = \frac{n}{2}\int_X(\sigma\overline\sigma)^{n-1}\alpha^2 + (1-n)\int_X\sigma^{n-1}\overline{\sigma}^n\alpha\int_X\sigma^n{\overline \sigma}^{n-1}\alpha,
$$
 where $\int_X$ is the cap product with the fundamental class.
\end{definition}  

This form $q_{X, \sigma}$ defines a quadratic form over $H^2(X, \mathbb R)$ by definition.  It is non-degenerate, and the signature of $q_{X, \sigma}$ is $(3, b_2-3)$, where $b_2 = \dim H^2(X, \mathbb R)$ is the second Betti number; see \cite[Theorem~2]{schwald2017fujiki}.  If we rescale $\sigma$ so that $\int_X(\sigma\overline{\sigma})^n = 1$, then $q_{X, \sigma}$ is independent of $\sigma$.  In this case, we call $q_X : = q_{X,\sigma}$ the \textit{Beauville--Bogomolov--Fujiki form}.

  We define the period domain of $X$ to be 
$$
\Omega : = \left\{[\sigma] \in \mathbb P(H^2(X, \mathbb C))\mid q_X(\sigma) = 0,~q_X(\sigma,\overline{\sigma}) > 0\right\}.
$$
  There is an associated period map 
$$
\rho\colon\Def^{\,\lt}(X) \lra \Omega
$$
 from the universal family of locally trivial deformations of $X$ which sends $t$ to $H^{2,0}(X_t)$, where $X_t$ is a locally trivial deformation of $X$ corresponding to a point $t \in \Def^{\,\lt}(X)$.  The period map $\rho$ is a local isomorphism by the \textit{local Torelli theorem}, see \cite[Proposition 5.5]{bakker2018global}, which leads to two immediate consequences.  First, $q_X$ satisfies the Fujiki relation: There is a positive constant $c$ such that 
\begin{equation}\label{regFujiki}
    q_X(\alpha)^n = c\int_X\alpha^{2n}; 
\end{equation}
see \cite[Proposition 5.15]{bakker2018global}.  Second, we may rescale the BBF form to get an \textit{integral} quadratic form on $H^2(X, \mathbb Z)$; see \cite[Lemma 5.7]{bakker2018global}.  We denote the corresponding integral lattice by $\Gamma : = (H^2(X, \mathbb Z), q_X)$.

The period map also satisfies many global properties.  We can view the period domain $\Omega$ as the moduli space of all weight 2 Hodge structures on $\Gamma$ which admit a quadratic form $q$ of signature $(3,b_2-3)$ which is positive-definite on the real space underlying $H^{2,0}\otimes H^{0,2}$.  We  then write $\Omega_\Gamma$ for the period domain.  If $X'$ is a locally trivial deformation of $X$ with $(H^2(X', \mathbb Z), q_{X'}) \cong \Gamma$, then we get a period map $p\colon \mathscr M' \to \Omega_\Gamma$ from the moduli space of $\Gamma$-marked locally trivial deformations of $X'$.  The orthogonal group $O(\Gamma)$ acts on $\mathscr M'$ and $\Omega_\Gamma$ by changing the marking.  For any connected component $\mathscr M$ of $\mathscr M'$, we denote by $\Mon(\mathscr M) \subset O(\Gamma)$  the image of the monodromy representation on second cohomology.

\begin{theorem}[\textit{cf.} \protect{\cite[Theorem 1.1(1)]{bakker2018global}}]\label{monodensity}
The monodromy group $\Mon(\mathscr M) \subset O(\Gamma)$ is of finite index.
\end{theorem}

\begin{remark} \label{monodromyremark}
  We will use the following interpretation of Theorem~\ref{monodensity}.  Let $\Mon(X)$ be the monodromy group associated to the connected component associated to a primitive symplectic variety $X$.  Since $\Mon(X) \subset O(\Gamma)$ is of finite index, the restriction $G_X : = \Mon(X)\cap \SO(\Gamma) \subset \SO(\Gamma)$ is also of finite index.  By the Borel density theorem, $G_X$ is therefore Zariski dense in $\SO(\Gamma_{\mathbb C})$.  
\end{remark}
    
    \subsection{Intersection Cohomology}
    
        Our proof of the LLV structure theorem for primitive symplectic varieties uses the Hodge theory of intersection cohomology. We review the ``Hodge--K\"ahler'' package for the intersection cohomology of a compact complex space, which follows generally from Saito's theory of mixed Hodge modules; see \cite{saito1990mixed}.  We refer the reader to \cite{de2005hodge} for an excellent treatment in the algebraic case.

    \subsubsection{Hodge theory of intersection cohomology}

The intersection cohomology is defined as the hypercohomology groups with respect to the intersection complex.  The intersection cohomology complex is the perverse sheaf underlying the unique \textit{Hodge module} determined by the constant variation of pure Hodge structures on $\mathbb Q_{X_{\mathrm{reg}}}[\dim X]$ over the regular locus of $X$; see \cite[Section~5.3]{saito1988modules}.  Writing $\mathcal{IC}_X$ for this complex, we have 
$$
\IH^k(X, \mathbb Q) : = \mathbb H^{k - \dim X}(X, \mathcal{IC}_X).
$$

\begin{proposition} \label{intprops}
Suppose that $X$ is a compact K\"ahler variety.

\begin{enumerate}
    \item\label{intprops-1} $($Decomposition theorem$)$.  If $f\colon Y \to X$ is a projective morphism, there is a non-canonical isomorphism 
$$
\IH^k(Y, \mathbb Q) \cong \IH^k(X, \mathbb Q)\oplus\bigoplus_\lambda \IH(X_\lambda, L_\lambda),
$$
 where the pairs $(X_\lambda, L_\lambda)$ consist of closed subvarieties with semisimple local systems $L_\lambda$.  In particular, $\IH^k(X, \mathbb Q) \subset \IH^k(Y, \mathbb Q)$.
    
    \item\label{intprops-2} $($Hodge decomposition$)$.  The cohomology groups $\IH^k(X, \mathbb Q)$ all carry pure Hodge structures of weight $k$.
    
    \item \label{ample HL} $($hard Lefschetz$)$.  Given an ample class $\alpha \in H^2(X, \mathbb Q)$ $($or, more generally, a K\"ahler class in $H^2(X, \mathbb R))$, there is a cup product map $L_\alpha\colon \IH^k(X, \mathbb Q) \to \IH^{k+2}(X, \mathbb Q)$ which produces isomorphisms 
$$
L_\alpha^j\colon \IH^{\dim X-j}(X, \mathbb Q) \overset{\lowsim}\lra \IH^{\dim X + j}(X, \mathbb Q)
$$
 for every $j$.
    
    \item \label{weight piece intersection map} For each $k$, there is a natural morphism $H^k(X, \mathbb Q) \to \IH^k(X, \mathbb Q)$ whose kernel is contained in $W_{k-1}H^k(X, \mathbb Q)$, the $\supst{(k-1)}$ piece of the weight filtration.
\end{enumerate}
\end{proposition}

\begin{remark} \label{remark intersection cohomology inclusion rational sings}
  If $X$ has rational singularities---such as a primitive symplectic variety---then $H^2(X, \mathbb C) \subset \IH^2(X, \mathbb C)$ since $H^2(X, \mathbb C)$ carries a pure Hodge structure by Proposition~\ref{intprops}(\ref{weight piece intersection map}).  
\end{remark}

We will use the following description of intersection cohomology throughout the paper.

\begin{lemma} \label{intcodim}
Suppose that $X$ is a projective variety with singular locus $\Sigma$.  If\, $\codim_X(\Sigma) \ge 3$, then $\IH^2(X, \mathbb C) \cong H^2(X\setminus \Sigma, \mathbb C)$.  
\end{lemma}  

\begin{proof}
This is \cite[Lemma 1]{durfee1995intersection}.  More generally, if $\codim_X(\Sigma_X) = d$, then $\IH^k(X, \mathbb C) \cong H^k(X\setminus \Sigma, \mathbb C)$ for $k < d$ and $k < \dim X$.
\end{proof}

Since the analytic topology of a singularity of a complex variety is preserved under locally trivial deformations, Lemma~\ref{intcodim} holds for any primitive symplectic variety, as a general locally trivial deformation is projective; see \cite[Corollary 6.10]{bakker2018global}.

    \subsection{Intersection Cohomology for Primitive Symplectic Varieties}

In order to prove the structure theorem for the total Lie algebra, we need to adapt the hard Lefschetz theorem for non-ample classes.  This can be done using the monodromy density theorem. 

\begin{proposition} \label{HLregclass}
Let $X$ be a primitive symplectic variety of dimension $2n$.  Let $\alpha \in H^2(X, \mathbb Q)$ be any non-isotropic class with respect to the BBF form $q_X$.  Then $\alpha$ satisfies hard Lefschetz: There is a cupping morphism $L_\alpha\colon \IH^k(X, \mathbb Q) \to \IH^{k+2}(X, \mathbb Q)$ which induces isomorphisms 
$$
L_\alpha^k\colon \IH^{2n-k}(X, \mathbb Q) \lra \IH^{2n+k}(X, \mathbb Q).
$$
 Conversely, any class $\alpha$ which satisfies hard Lefschetz is non-isotropic.
\end{proposition}

\begin{proof}
First, we note that there is a cupping morphism $L_\alpha\colon \IH^k(X, \mathbb Q) \to \IH^{k+2}(X, \mathbb Q)$ for \textit{any} $\alpha \in \IH^2(X, \mathbb Q)$, which agrees with the usual cup product of Proposition~\ref{intprops}(\ref{ample HL}) when $\alpha$ is a K\"ahler form. 
 We follow \cite[Section~4.4]{de2005hodge}.  By the decomposition theorem Proposition~\ref{intprops}(\ref{intprops-1}), any class $\alpha \in \IH^2(X, \mathbb Q)$ can be lifted to a class $\tilde \alpha \in H^2(\tilde X, \mathbb Q)$ for a resolution of singularities $\pi\colon\tilde X \to X$.  Consider the isomorphism 
$$
H^2(\widetilde X, \mathbb Q) \cong \Hom_{D(X)}\left(\mathbb Q_{\widetilde X}, \mathbb Q_{\widetilde X}[2]\right),
$$
 where $D(X)$ is the full subcategory of the bounded derived category $D^b(X)$ of constructible sheaves which are \textit{cohomologically constructible}; see \cite[Definition 3.3.1]{de2005hodge}.  In particular, we obtain a map 
$$
\mathbf R\pi_*\tilde \alpha\colon \mathbf R\pi_*\mathbb Q_{\widetilde X} \lra \mathbf R\pi_*\mathbb Q_{\widetilde{X}}[2].
$$
  Now this construction is consistent with any $p$-splitting of $\mathbf R\pi_*\mathbb Q_{\widetilde X}$, in the sense of \cite[Definition~4.3.1]{de2005hodge}.  In particular, this induces a map $\alpha\colon \mathcal{IC}_X \to \mathcal{IC}_X[2]$ via the decomposition theorem.  This gives the desired morphism $L_\alpha$ upon taking hypercohomology. 

Let $\Gamma = (H^2(X, \mathbb Z), q_X)$, and consider the monodromy group $\Mon(X)$ of $\Gamma$-marked primitive symplectic varieties deformation-equivalent to $X$.  Then $\Mon(X)$ is a finite-index subgroup of $O(\Gamma)$ by Theorem~\ref{monodensity}.  Then consider  the subgroup 
$$
G_X = \Mon(X)\cap \SO(\Gamma),
$$
 which is a Zariski-dense subset of $\SO(\Gamma_{\mathbb C})$ by Remark~\ref{monodromyremark}.

Let $\ker \alpha^k = \ker(L_\alpha^k\colon \IH^{2n-k}(X, \mathbb Q) \to \IH^{2n+k}(X, \mathbb Q))$.  By Poincar\'e duality, it is enough to show that $\dim \ker \alpha^k = 0$.  Now for any $g \in G_X$, it follows that $\dim \ker \alpha^k = \dim\ker (g\cdot \alpha)^k$, as monodromy preserves cup product.  If $G_{X, \alpha}$ is the $G_{X, \alpha}$-orbit of $\alpha$, we get the constant map $G_{X, \alpha}^\circ \to \mathbb Z$ given by $g \mapsto \dim \ker  (g\cdot \alpha)^k$.  But $G_X$ is Zariski dense, and therefore the constant map must extend to a constant map over the full $\SO(\Gamma_{\mathbb C})$-orbit of $\alpha$; but this orbit necessarily contains the class of an ample divisor on some locally trivial deformation, which satisfies hard Lefschetz by Proposition~\ref{intprops}(\ref{ample HL}).  In particular, $\dim\ker \alpha^k = 0$ for every $\alpha$.

Conversely, if $\alpha$ satisfies hard Lefschetz, then the Fujiki relation (\ref{regFujiki}) implies $q_X(\alpha) \ne 0$.
\end{proof}

We also need to move the intersection cohomology around in locally trivial deformations. 

\begin{proposition} \label{intloc}
The intersection cohomology groups $\IH^k(X_t, \mathbb C)$ for $t \in \Def^{\,\lt}(X)$ form a local system.
\end{proposition}

\begin{proof}
As vector spaces, intersection cohomology is completely determined by the structure $X_t$ admits as a stratified pseudomanifold by Deligne's construction; see \cite[Theorem 3.5]{goresky1983intersection}.  The claim follows since locally trivial deformation are real analytically trivial; see \cite[Proposition 5.1]{amerik2021contraction}. 
\end{proof}

    \subsection{Some Mixed Hodge Structures} \label{2.4} The Hodge theory of the intersection cohomology of primitive symplectic varieties with isolated singularities can be completely described by the Hodge theory of its regular locus, see Section~\ref{subsection symplectic Hard Lefschetz}, and so we review the relevant parts of Deligne's mixed Hodge structure on the (compactly supported) cohomology of smooth varieties.  Deligne's original treatment holds for algebraic varieties, but the same results hold in the K\"ahler setting by \cite{fujiki1980duality}.
    
    \subsubsection{Cohomology of the regular locus}
    
    Let $U$ be a smooth K\"ahler variety, and let $\widetilde{X}$ be a smooth compactification of $U$ such that the complement $\widetilde{X} \setminus U \cong E$ is a simple normal crossing (snc) divisor.  A fundamental result of Deligne \cite[Proposition 3.1.8]{deligne1971theorie} states that the cohomology of $U$ can be identified with the hypercohomology of the complex of logarithmic forms on the pair $(\widetilde{X}, E)$.  Specifically, let $\Omega_{\widetilde{X}}^p(\log E)$ be the sheaf of logarithmic $p$-forms, and let $\Omega_{\widetilde{X}}^\bullet(\log E)$ the complex of logarithmic forms.  Then there is an isomorphism 
$$
H^k(U, \mathbb C) \cong \mathbb H^k\left(\widetilde{X}, \Omega_{\widetilde{X}}^\bullet(\log E)\right),
$$
 which induces two filtrations on $H^k(U, \mathbb C)$.  The first is the naive filtration associated to the complex $\Omega_{\widetilde{X}}^\bullet(\log E)$: 
$$
F^pH^k(U, \mathbb C) = \im\left(\mathbb H^k\left(\widetilde{X}, \tau_{\ge p}\Omega_{\widetilde{X}}^\bullet(\log E)\right) \lra \mathbb H^k\left(\widetilde{X}, \Omega_{\widetilde{X}}^\bullet(\log E)\right)\right), 
$$
 where $\tau_{\ge p}\Omega_{\widetilde{X}}^\bullet(\log E)$ is the complex $\Omega_{\widetilde{X}}^\bullet(\log E)$ truncated in degree greater than or equal to $p$. The second filtration is induced at the level of sheaves: There is an increasing filtration on logarithmic $p$-forms given by 
$$
W_l\Omega_{\widetilde{X}}^p(\log E) = \Omega_{\widetilde{X}}^l(\log E)\wedge \Omega_{\widetilde{X}}^{p-l}
$$
 descending to $\mathbb H^k(\widetilde X, \Omega_{\widetilde X}^\bullet (\log E)$.  These filtrations correspond to the canonical mixed Hodge structure on $H^k(U, \mathbb Q)$---$W_\bullet$ is the complexification of the rational weight filtration, and $F^\bullet$ is the Hodge filtration.
    
    In particular, there is a non-canonical decomposition 
    \begin{equation}\label{noncanonical log}
        H^k(U, \mathbb C) \cong \bigoplus_{p + q = k} H^q\left(\widetilde{X}, \Omega_{\widetilde{X}}^p(\log E)\right).
    \end{equation}
    We note that the above construction applies to the regular locus $U$ of any proper K\"ahler variety $X$, where we may choose the smooth compactification to be a log-resolution $\widetilde{X}$ of $X$.

    \subsubsection{Compactly supported cohomology of the regular locus} Deligne \cite{deligne1971theorie, deligne1974theorie} also shows that the compactly supported cohomology of an algebraic variety carries a pure Hodge structure.

    The compactly supported cohomology also carries a mixed Hodge structure.  The easiest way to see this is to use Poincar\'e duality.  Indeed, the Poincar\'e isomorphisms 
$$
    H^k_c(U, \mathbb C) \cong \left(H^{2\dim U -k}(U, \mathbb C)\right)^*\otimes \mathbb C(-\dim U)
$$
 are  isomorphisms of mixed Hodge structures.  Noting that 
$$
H^q\left(\Omega_{\widetilde{X}}^p(\log E)\right) \cong H^{\dim \widetilde{X}-q}\left(\widetilde{X}, \Omega_{\widetilde{X}}^{\dim \widetilde{X}-p}(\log E)(-E)\right)^*
$$
 by Serre duality, this completely describes the mixed Hodge structure on $H^k_c(U, \mathbb Q)$.  In particular, we have another non-canonical splitting 
    \begin{equation} \label{noncanonical log zero}
        H_c^k(U, \mathbb C) \cong \bigoplus_{p + q = k}H^q\left(\widetilde{X}, \Omega_{\widetilde{X}}^p(\log E)(-E)\right).
    \end{equation}
    
    \subsubsection{Cohomology of a simple normal crossing divisor}

    Finally, let us consider an snc divisor $E = \sum_{i=1}^r E_i$ with irreducible components $E_i$.  For any subset $J \subset \left\{1,\ldots,r\right\}$, 
    write $E_J = \bigcap_{j \in J} E_j$ and let
    \begin{equation} \label{p-fold intersections}
        E_{(p)} : = \coprod_{|J| = p} E_J
    \end{equation} 
    be the $p$-fold intersections of the components.  For each $p$, the various inclusion maps $E_{(p)} \hookrightarrow E_{(p-1)}$ induce a simplicial map $H^k(E_{(p-1)}, \mathbb Q) \to H^k(E_{(p)}, \mathbb Q)$.  From this we obtain a complex 
$$
    0 \lra H^k(E_{(1)}, \mathbb Q) \overset{\delta_1}\lra H^k(E_{(2)}, \mathbb Q) \overset{\delta_2}\lra\cdots \xrightarrow{\delta_{p-1}} H^k(E_{(p)}, \mathbb Q) \overset{\delta_p}\lra\cdots
    $$
 which, for each $k$, computes graded pieces of the weight filtration on $H^k(E, \mathbb C)$: 
$$
\gr_k^WH^{k+r}(E, \mathbb C) = \ker\delta_{r+1}/\im \delta_r.
$$
 The $(p,q)$-Hodge pieces are obtained by applying the functor $H^q(\Omega^p)$ to the complex and taking cohomology.  Precisely, there is an induced complex with morphisms 
$$
\delta^{p,q}_r\colon H^q\left(E_{(r)}, \Omega_{E_{(r)}}^p\right) \lra H^q\left(E_{(r+1)}, \Omega_{E_{(r+1)}}^p\right),
$$
 and 
    \begin{equation}\label{exceptionalhodge}
        \gr_k^WH^{k + r}(E, \mathbb C)^{p,q} = \ker \delta_{r+1}^{p,q}/\im~\delta_r^{p,q}.
    \end{equation}
    
    \subsection{A Result on Bimeromorphic Morphisms of Symplectic Varieties}
    
    Recall that a morphism $\phi\colon Z\to X$ is \textit{semismall} if $\dim Z \times_X Z \le \dim X$.  A result of Kaledin states that symplectic resolutions are semismall; see \cite[Lemma 2.11]{kaledin2006symplectic}.  We want to extend this result to bimeromorphic morphisms of singular symplectic varieties, which will control the geometry of such morphisms.

\begin{definition}
Let $X$ and $Z$ be normal and $\mathbb Q$-Gorenstein varieties.  A birational morphism $\phi\colon Z \to X$ is \textit{crepant} if $\phi^*K_X = K_Z$.
\end{definition}

\begin{lemma}\label{symcrepant}
If\, $X$ is a $($primitive$)$ symplectic variety and $\phi\colon Z \to X$ is a crepant birational morphism from a normal complex variety $Z$, then $Z$ is also a $($primitive$)$ symplectic variety.
\end{lemma}

\begin{proof}
By taking a common resolution of singularities of $Z$ and $X$, we can see that the symplectic form on $X_{\mathrm{reg}}$ extends to a global reflexive 2-form $\sigma_Z$ on $Z$.  Since $K_Z = \phi^*K_X = \mathscr O_Z$, then $\sigma_Z$ must define a holomorphic symplectic form on $Z_{\mathrm{reg}}$, which therefore extends to a holomorphic 2-form on the common resolution (and therefore any resolution) of singularities.  If $X$ is primitive symplectic, it is clear that $Z$ is also a compact K\"ahler variety, $H^1(\mathscr O_Z) = H^1(\mathscr O_X) = 0$, and $H^0(Z, \Omega_Z^{[2]}) = \mathbb C\cdot \sigma_Z$.
\end{proof}

\begin{lemma}
Let $\phi\colon Z \to X$ be a crepant morphism of primitive symplectic varieties.  Then the restriction of the class $\sigma_Z$ to $\phi^{-1}(x)$ is $0$ for every $x$.
\end{lemma}

\begin{proof}
Note that if $X$ is primitive symplectic, then so is $Z$.  The result then follows from Hodge theory: By assumption, the class of the symplectic form $\phi^*\sigma_X$ extends to a symplectic form $\sigma_Z$ on $Z$, as both must extend to any common resolution of singularities.  Therefore, its image under the morphism $H^2(Z, \mathbb C) \to H^0(X, R^2\phi_*\mathbb C)$ is zero.  The result then follows by proper base change.
\end{proof}

\begin{proposition}\label{crepantsemismall}
If $X$ is a primitive symplectic variety and $\phi\colon Z \to X$ is a crepant morphism, then $\phi$ is semismall.
\end{proposition}

\begin{proof}
  Let $\left\{X_i\right\}$ and $\left\{Z_j\right\}$ be the Kaledin stratifications of $X$ and $Z$ (see Theorem~\ref{symstrat}).
  We will show that \begin{equation} \label{equation semismall inequality}
    2 \dim \phi^{-1}(x) + \dim X_i \le \dim X.
\end{equation} To prove this, we follow \cite[Lemma 2.11]{kaledin2006symplectic} in the case of symplectic resolutions.  Kaledin's proof, built upon work of Wierzba \cite{wierzba2003contractions} and Namikawa \cite{namikawa2000extension}, uses the symplectic structure of a smooth symplectic variety $Z$ to show that the tangent spaces $T_z\phi^{-1}(x)$ and $T_z\phi^{-1}(X_i)$ are mutually orthogonal, which immediately gives (\ref{equation semismall inequality}).  We will therefore consider how the fibers intersect with the various smooth and symplectic strata $Z_j^\circ$.    

Let $x \in X$, and let $X_i$ be the stratum for which $x \in X_{i}^\circ$.  Write $E_x = \phi^{-1}(x)$ and $E = \phi^{-1}(X_i)$.  For every $z \in E_{\mathrm{reg}}$, there is a $j$ such that $z \in Z_{j}^\circ$.  There is then a map $\widehat \phi_x\colon \widehat{Z_{j}^\circ}_z \to \widehat{X_{i}^\circ}_x$ induced from the map $\widehat \phi\colon\widehat Z_z \to \widehat X_x$ via the product decomposition.  Consider the commutative diagram 

\[ \begin{tikzcd}
\widehat{E_{\mathrm{reg}}}_z \arrow{r} \arrow[swap]{d} & \widehat{Z_j^{\circ}}_z \arrow{r} \arrow{dl}{\widehat{\phi}_x} & \widehat Z_z \arrow{d}{\widehat{\phi}} \\%
\widehat{X_{i}^\circ}_x \arrow[rr]&  & \widehat X_x\rlap{,}
\end{tikzcd}
\] where $\widehat{E_{\mathrm{reg}}}_z$ is the completion of $E_{\mathrm{reg}}$ at $z$.  The varieties in the diagram all have at worst rational singularities, and the $\widehat{X_i^\circ}_x$, $\widehat{Z_j^\circ}_z$ are symplectic varieties.  By reflexive pullback, we see that the symplectic form $\widehat{\sigma}_i$ on $\widehat{X_i^\circ}_x$ pulls back to the restriction $\widehat{\sigma}_j|_{\widehat{E_{\mathrm{reg}}}_z}$, where $\widehat\sigma_j$ is the symplectic form on $\widehat{Z_j^\circ}_z$, which is just the restriction of the symplectic form $\widehat \sigma_Z$ on $\widehat Z_z$.  Since $\sigma_Z$ vanishes on the fibers $\phi^{-1}(x)$ (as a cohomology class), then $\widehat\sigma_Z$ vanishes on the fibers of $\widehat\phi$. Since $\widehat{Z_j^\circ}_z$ is smooth and symplectic, we may assume the tangent spaces $T_zE_x$ and $T_z(E_{\mathrm{reg}})$ are mutually orthogonal with respect to the symplectic form $\sigma_j$ on $Z_j^\circ$ after passing to a small open neighborhood of $z$.  Therefore, 
$$
\dim E_x \le \dim X - \dim E \le \dim X - \dim E_x - \dim X_i.
$$
The second inequality 
above follows from local product structure of $X$ along the smooth stratum $X_i^\circ$.  This clearly agrees with (\ref{equation semismall inequality}), and so $\phi$ is semismall.    
\end{proof}

    \subsection{A $\mathbb Q$-Factoriality Criterion}
    
    Let $X$ be a projective variety.  We say that $X$ is $\mathbb Q$-factorial if every Weil divisor is $\mathbb Q$-Cartier.  Equivalently, see \cite[Proposition~12.1.6]{kollar1992classification}, the variety $X$ is $\mathbb Q$-Factorial if and only if for a resolution of singularities $\pi\colon\widetilde{X} \to X$ with exceptional divisor $E = \sum E_i$, \begin{equation}\label{Kollar Mori Q fact}
      \im\left(H^2(\widetilde{X}, \mathbb Q) \lra H^0(X, R^2\pi_*\mathbb Q)\right) = \im\left(\bigoplus_i \mathbb Q[E_i] \lra H^0(X, R^2\pi_*\mathbb Q)\right).  
    \end{equation}

    \begin{proposition} \label{Qcrit}
    If\, $X$ is a terminal primitive symplectic variety, then $X$ is $\mathbb Q$-factorial if and only if the natural inclusion $H^2(X, \mathbb C) \hookrightarrow \IH^2(X, \mathbb C)$ is an isomorphism.
    \end{proposition}

    \begin{proof}
    Since $X$ is terminal, we have an isomorphism $\IH^2(X, \mathbb C) \xrightarrow{\lowsim} H^2(U, \mathbb C)$ by Lemma~\ref{intcodim} and Proposition~\ref{symterm}.  The mixed Hodge structure on $H^2(U, \mathbb Q)$ is then pure, and the logarithmic Hodge-to-de Rham spectral sequence computes the Hodge filtration; see \cite[Section~(9.2.3)]{deligne1974theorie}.  In particular,
$$
    H^2(U, \mathbb C) \cong H^0\left(\widetilde{X}, \Omega_{\widetilde{X}}^2(\log E)\right) \oplus H^1\left(\widetilde{X}, \Omega_{\widetilde{X}}^1(\log E)\right) \oplus H^2\left(\widetilde{X}, \mathscr O_{\widetilde{X}}\right).
$$
 Note that the inclusion $H^2(U, \mathbb C) \subset H^2(\widetilde{X}, \mathbb C)$ differs only on the $(1,1)$-pieces, as $X$ has rational singularities and $\pi_*\Omega_{\widetilde{X}}^2$ is reflexive.  Consider the long exact sequence
    \begin{equation*}
    \begin{split}
        0 \lra H^0\left(\widetilde{X}, \Omega_{\widetilde{X}}^1\right) \lra H^0\left(\widetilde{X}, \Omega_{\widetilde{X}}^1(\log E)\right) & \lra \bigoplus H^0\left(\mathscr O_{E_i}\right) \\ & \lra H^1\left(\widetilde{X}, \Omega_{\widetilde{X}}^1\right) \lra H^1\left(\widetilde{X}, \Omega_{\widetilde{X}}^1(\log E)\right)
    \end{split}
    \end{equation*}
    By \cite[Corollary 1.8]{kebekus2021extending}, the first morphism is an isomorphism. The last morphism is surjective, as $H^1(\widetilde{X}, \Omega_{\widetilde{X}}^1(\log E)) \xrightarrow{\lowsim} \IH^{1,1}(X)$, and so we have a surjection by the decomposition theorem.  Therefore,  
$$
\dim H^{1,1}\left(\widetilde{X}\right) - \dim H^{1,1}(U) = \sum \dim H^0(E_i).
$$
 Now consider the inclusion $H^2(X, \mathbb C) \subset H^2(\widetilde{X}, \mathbb C)$.  Again, these vector spaces differ only in the $(1,1)$-classes.  The Leray spectral sequence induces the exact sequence 
$$
0 \lra H^1\left(X, \pi_*\Omega_{\widetilde{X}}^1\right) \lra H^1\left(\widetilde{X}, \Omega_{\widetilde{X}}^1\right) \lra H^0\left(X, R^1\pi_*\Omega_{\widetilde{X}}^1\right).
$$
 Thus the cokernel of the inclusion is exactly $\im(H^2(\widetilde{X}, \Omega_{\widetilde{X}}^1) \to H^0(R^1\pi_*\Omega_{\widetilde{X}}^1))$.  But the symplectic form clearly gets killed in the morphism $H^2(X, \mathbb C) \to H^0(R^2\pi_*\mathbb C)$, so by Hodge theory 
$$
\im\left(H^1\left(\widetilde{X}, \Omega_{\widetilde{X}}^1\right) \lra H^0\left(R^1\pi_*\Omega_{\widetilde{X}}^1\right)\right) = \im\left(H^2\left(\widetilde{X}, \mathbb C\right) \lra H^0\left(R^2\pi_*\mathbb C\right)\right).
$$
 We now see  by (\ref{Kollar Mori Q fact}) that $X$ is then $\mathbb Q$-factorial if and only if $H^2(X, \mathbb C) \cong \IH^2(X, \mathbb C)$.
    \end{proof}

    The results of the last two sections will be used to allow us to pass from a primitive symplectic variety to a bimeromorphic model with at worst $\mathbb Q$-factorial singularities to prove the results of this paper.  Recall that a \textit{$\mathbb Q$-factorial terminalization} of $X$ is a crepant morphism $\phi\colon Z \to X$ from a $\mathbb Q$-factorial terminal variety~$Z$.  Such morphisms exist for projective varieties, see \cite{birkar2010existence}, and very general members of a locally trivial deformation of a primitive symplectic variety with $b_2 \ge 5$; see \cite[Corollary 9.2]{bakker2018global}.  

    \begin{corollary}\label{Qfactorialterminalsemismall}
    Suppose that $X$ is a primitive symplectic variety and $\phi\colon Z \to X$ a $\mathbb Q$-factorial terminalization.  Then $\phi$ is semismall, and there is a canonical injection $\IH^2(X, \mathbb C) \hookrightarrow H^2(Z, \mathbb C)$.
    \end{corollary}
    
    \subsection{Background on the LLV Algebra} \label{2.7}
    
    To end this section, we review the construction of the LLV algebra and the LLV structure theorem for hyperk\"ahler manifolds.  This will also allow us to indicate the necessary pieces for an algebraic proof of the LLV structure theorem for intersection cohomology.
    
    The \textit{total Lie algebra} of a compact complex variety $Y$ is the Lie algebra generated by all Lefschetz operators corresponding to hard Lefschetz, or HL, classes $\alpha$: 
    \begin{equation}
        \mathfrak g_{\mathrm{tot}}(Y) : = \langle L_\alpha, \Lambda_\alpha:~\alpha~\mathrm{is~HL}\rangle.
    \end{equation}
    Here, $L_\alpha$ is the cupping operator with respect to the $(1,1)$-class $\alpha$, and $\Lambda_\alpha$ is its \textit{dual Lefschetz operator}.  
    
     When $X$ is a hyperk\"ahler manifold, we define the \textit{LLV algebra} to be its total Lie algebra, and we write $\mathfrak g = \mathfrak g_{\mathrm{tot}}(X)$.  
     
     Let $g$ be a hyperk\"ahler metric on $X$.  Associated to $(X,g)$ are three differential forms $\omega_1,\omega_2, \omega_3$ which are K\"ahler forms with respect to $g$.  Let $W_g : = \langle \omega_1,\omega_2, \omega_3\rangle$ be the three-space associated to this metric, and in fact a positive three-space with respect to the BBF form $q_X$, which means that $q_X$ is positive-definite on $W_g$.  Consider the algebra $\mathfrak g_g = \langle L_{\omega_i}, \Lambda_{\omega_i}:~\omega_i \in W_g\rangle$.  Verbitsky \cite{verbitsky1990action} showed that this algebra already has a structure generalizing the hard Lefschetz theorem for K\"ahler manifolds.
     
     \begin{theorem}\label{so(4,1)}
     
     We have $\mathfrak g_g \cong \mathfrak{so}(4,1)$.
     
     \end{theorem}
     
     The proof is as follows.  Let $2n$ be the dimension of $X$.  The dual Lefschetz operators $\Lambda_{\omega_i}$ are uniquely determined by the property that $[L_{\omega_i}, \Lambda_{\omega_i}](\alpha) = H(\alpha) =  (k-2n)\alpha$, where $\alpha \in H^k(X, \mathbb C)$.  Therefore, one can show that each $\Lambda_{\omega_i}$ is the adjoint of $L_{\omega_i}$ with respect to the Hodge star operator, which depends only on the K\"ahler structure $(X,g)$.  This then implies  that the dual Lefschetz operators commute: 
$$
\left[\Lambda_{\omega_i}, \Lambda_{\omega_j}\right] = 0.
$$
 This is the main geometric input.  The fact that $\mathfrak g_g \cong \mathfrak{so}(4,1)$ follows from this geometric input, the hard Lefschetz theory for the $\omega_i$, and the following additional commutator relations:
     \begin{equation} \label{verbitsky commutator relations}
         \begin{split}
            & K_{ij} = -K_{ji}, \quad   \left[K_{ij}, K_{jk}\right] = 2K_{ik}, \quad   \left[K_{ij}, H\right] = 0,  \\ &\left[K_{ij}, L_{\omega_j}\right] = 2L_{\omega_i}, \quad \left[K_{ij}, \Lambda_{\omega_j}\right] = 2\Lambda_{\omega_i}, \quad  \left[K_{ij}, L_{\omega_k}\right] = \left[K_{ij}, \Lambda_{\omega_k}\right] = 0,\  i,j \ne k,
         \end{split}
     \end{equation}
     where $K_{ij} := [L_{\omega_i}, \Lambda_{\omega_j}]$.  Verbitsky observed in \cite{verbitsky1990action} that $K_{ij}$ acts as the Weil operator with respect to the complex structure induced by $\omega_k$ (and similarly for $K_{jk}$ and $K_{ki}$).  The commutator relations (\ref{verbitsky commutator relations}) follow from this key observation. 
     
     Verbitsky's approach to the LLV structure theorem in his thesis \cite{verbitsky1996cohomology} was to then look at $\mathfrak g$ as the algebra generated by the $\mathfrak g_g$ by varying the hyperk\"ahler metric via the period map.  The main technical input, observed by Verbitsky and also Looijenga--Lunts in \cite{looijenga1997lie}, is that \textit{all} the dual Lefschetz operators commute whenever they are defined.  The following are the necessary pieces for obtaining the LLV structure theorem. 
     
     \begin{theorem}[\textit{cf.} \protect{\cite[Theorem 4.5]{looijenga1997lie}}, \protect{\cite[Proposition 1.6]{verbitsky1996cohomology}}]\label{LLVpieces}
     Let $X$ be a hyperk\"ahler manifold.  
     \begin{enumerate}
         \item \label{hyperdecomp} For any two classes $\alpha, \beta \in H^2(X, \mathbb R)$ satisfying hard Lefschetz, we have $[\Lambda_\alpha, \Lambda_\beta] = 0$.  Thus, the LLV algebra only exists in degrees $2$, $0$, and $-2$, and we get an eigenspace decomposition 
$$
\mathfrak g \cong \mathfrak g_2\oplus \mathfrak g_0\oplus \mathfrak g_{-2}
$$
 with respect to the weight operator $H$ acting as the adjoint.
         
         \item \label{hyperkahler g2 structure} There are canonical isomorphisms $\mathfrak g_{\pm 2} \cong H^2(X, \mathbb R)$ of\, $\mathfrak g$-modules.
         
         \item \label{hyperkahler semisimple} There is a decomposition $\mathfrak g_0 = \overline{\mathfrak g}\oplus \mathbb R\cdot H$, where $\overline{\mathfrak g}$ is the semisimple part.  Moreover, $\overline{\mathfrak g} \cong \mathfrak{so}(H^2(X, \mathbb R), q_X)$, and $\overline{\mathfrak g}$ acts on $H^*(X, \mathbb R)$ by derivations.
     \end{enumerate}
     \end{theorem}
     
     Theorem~\ref{LLVpieces}(\ref{hyperdecomp}) follows from the fact that the collection of positive three-spaces $W_g$ forms a dense open subset of the Grassmanian of three-spaces in $H^2(X, \mathbb R)$, whence the commutativity of the dual Lefschetz operators follows from the commutativity over the various $W_g$, and local Torelli.  The decomposition holds since the direct sum $\mathfrak g_2\oplus \mathfrak g_0\oplus \mathfrak g_{-2}$ is a Lie subalgebra of $\mathfrak g$, which follows from Theorem~\ref{LLVpieces}(\ref{hyperkahler g2 structure},~\ref{hyperkahler semisimple}).  Indeed, the openness of the space of positive three-spaces in $H^2(X, \mathbb R)$ implies that the semisimple part is generated by the commutators $[L_\alpha, \Lambda_\beta]$ for HL classes $\alpha, \beta$.  If $\alpha, \beta$ come from a positive three-space $W_g$, then Verbitsky \cite[Lemma 2.2]{verbitsky1990action} shows that
     \begin{equation} \label{Weilop}
     \left[L_\alpha, \Lambda_\beta\right](x) = i(p-q)x, 
     \end{equation}
     \noindent where $x$ is a $(p,q)$-form with respect to the metric $g$.  But this certainly acts on $H^*(X, \mathbb C)$ by derivations, and so \textit{every} commutator $[L_\alpha, \Lambda_\beta]$ acts via derivations on the cohomology ring.  This fact implies that $[\overline{\mathfrak g}, \mathfrak g_{\pm 2}] \subset \mathfrak g_{\pm 2}$; since $\mathfrak g_{2}$ and $\mathfrak g_{-2}$ are abelian, this gives the eigenvalue decomposition.  In order to prove that $\overline{\mathfrak g} \cong \mathfrak{so}(H^2(X, \mathbb R), q_X)$, we note that the semisimple part preserves cup product via derivation; the Fujiki relation~\eqref{regFujiki} then implies that $\overline{\mathfrak g}$ preserves $q_X$, and so $\overline{\mathfrak g} \subset \mathfrak{so}(H^2(X, \mathbb R), q_X)$.  The surjectivity follows by varying the $\mathfrak{so}(4,1)$-actions in the period domain since these generate the full $\mathfrak{so}(H^2(X, \mathbb C), q_X)$.

We remark that our theorems for the LLV algebra are stated with rational coefficients, while the works of Looijenga--Lunts and Verbitsky work over real or complex coefficients.  In the smooth case, it was observed in \cite{green2019llv} that the LLV structure theorem holds over $\mathbb Q$, as the operators are all rationally defined.  The same will hold for the singular version of the LLV algebra for the intersection cohomology of primitive symplectic varieties.

\section{Symplectic Symmetry on Intersection Cohomology} \label{3} We prove that the canonical Hodge structure on $\IH^*(X, \mathbb Q)$ inherited from the symplectic form $\sigma$ satisfies the symplectic hard Lefschetz theorem, one of the main inputs in our algebraic proof of the LLV structure theorem. 

\subsection{Degeneration of Hodge-to-de Rham on the Regular Locus}  The first piece we need in proving the symplectic hard Lefschetz theorem is the degeneration of Hodge-to-de Rham on the regular locus.  What is surprising here is that the degeneration holds with no restriction on the singularities of $X$.  We adapt a well-known trick to identify the Hodge-to-de Rham spectral sequence with the logarithmic Hodge-to-de Rham spectral sequence associated to a log-resolution of singularities. 

\begin{theorem} \label{theorem hodge to de rham}
Suppose $X$ is a proper symplectic variety of dimension $2n$ with regular locus $U$ and smooth singular locus $\Sigma$.  If the singular locus $\Sigma$ of\, $X$ is smooth, then the Hodge-to-de Rham spectral sequence 
$$
E_1^{p,q} = H^q\left(U, \Omega_U^p\right) \Longrightarrow H^{p+q}(U, \mathbb C)
$$
 degenerates at $E_1$ for $p + q < 2n-1$.
\end{theorem}

\begin{proof}
Consider the logarithmic Hodge-to-de Rham spectral sequence 
$$
E_1^{p,q}  = H^q\left(\widetilde{X}, \Omega_{\widetilde{X}}^p(\log E)\right) \Longrightarrow H^{p+q}(U, \mathbb C)
$$
 corresponding to a log-resolution of singularities $\pi\colon\widetilde{X} \to X$ with exceptional divisor $E$, which degenerates at $E_1$ for all $p,q$; see \cite[Section~(9.2.3)]{deligne1974theorie}.  It is enough to show that the restriction morphisms $H^q(\widetilde{X}, \Omega_{\widetilde{X}}^p(\log E)) \to H^q(U, \Omega_U^p)$ are isomorphisms for $p + q < 2n-1$.

The restriction morphisms fit inside a long exact sequence 
$$
\cdots \lra H^q_E\left(\widetilde{X}, \Omega_{\widetilde{X}}^p(\log E)\right) \lra H^q\left(\widetilde{X}, \Omega_{\widetilde{X}}^p(\log E)\right) \lra H^q(U, \Omega_U^p) \lra\cdots. 
$$
 By local duality, there is an isomorphism 
$$
H_E^q\left(\widetilde{X}, \Omega_{\widetilde{X}}^p(\log E)\right)^* \cong H^{2n-q}\left(\widetilde{X}_E, \Omega_{\widetilde{X}}^{2n-p}(\log E)(-E)\right),
$$
 where $\widetilde{X}_E$ is the completion of $\widetilde{X}$ along $E$.  Now consider  the Leray spectral sequence 
$$
E_2^{r,s} = H^r\left(X_\Sigma, R^s\pi_*\Omega_{\widetilde{X}}^{2n-p}(\log E)(-E)\right) \Longrightarrow H^{r+s}\left(\widetilde{X}_E, \Omega_{\widetilde{X}}^{2n-p}(\log E)(-E)\right).
$$
  It is enough to show that $E_2^{r,s}$ vanishes in the range $r + s > 2n+ 1$, which we note clearly holds when $\dim \Sigma = 0$.  To prove this, we claim that 
$$
R^q\pi_*\Omega_{\widetilde X}^p(\log E)(-E)_x = 0
$$
 for $p + q > 2n - \dim \Sigma_x,$ where $\Sigma_x$ is the connected component of $\Sigma$ which contains $x$.  By Proposition~\ref{symstrat}, there is a product decomposition 
$
\widehat{X}_x \cong Y_x\times \widehat{\Sigma}_x
$,
 where $Y_x$ is a symplectic variety with isolated singularities.  Since the claimed vanishing is local, we may assume that the log-resolution of singularities is $\pi = \pi_x\times \id$,\footnote{Note that the complex $\mathbf R\pi_*\Omega_{\tilde X}^p(\log E)(-E)$ is, up to quasi-isomorphism, independent of the choice of log-resolution $\pi$.} where $\pi_x\colon \widetilde{Y_x} \to Y_x$ is a log-resolution of singularities of $Y_x$.  It follows that $R^q\pi_*\Omega_{\widetilde X}^p(\log E)(-E) = 0$ if $R^q(\pi_x)_*\Omega_{\widetilde{Y_x}}^p(\log E_x)(-E_x) = 0$, where $E_x$ is the exceptional divisor of $\pi_x$. 

To conclude, consider the terms $E_2^{r,s}$ for $r + s > 2n + 1$.  On the one hand, $E_2^{r,s} = 0$ for $r > \dim \Sigma$, so we may assume otherwise.  On the other hand, if we write $\Sigma = \coprod_j \Sigma_j$, where $\Sigma_j \subset \Sigma$ consists of the smooth components of $\Sigma$ of dimension $j$, the preceding argument implies that $\Supp~R^s\pi_*\Omega_{\widetilde X}^{2n-p}(\log E)(-E)) \cap \Sigma_j = \emptyset$ if $s + (2n-p) > 2n-j$. Putting this together, we can see that $E_2^{r,s} = 0$ implies that $r + s \le 2n$.  
\end{proof}

An early draft claimed that this degeneration held for general primitive symplectic varieties---the following example indicates the gap in the previous argument when $\Sigma$ is not smooth; a new idea will be needed to extend this degeneration.

\begin{example}
Let $X$ be a primitive symplectic variety, let $\widehat X_x$ and $\widehat{X_i^\circ}_x$ be the completions of $X$ and $X_i^\circ$, respectively, at $x$, and let $Y_x$ be a symplectic variety.  Then 
$$
R^q\pi_*\Omega_{\tilde X}^p(\log E)(-E)_x = 0, \quad p+ q > \max\left\{2n, 2n-\dim X_i^\circ + p\right\}.
$$
  Consider the product decomposition 
$
\widehat{X}_x \cong Y_x \times \widehat{X_i^\circ}_x
$.
  We note the claim is local and independent of the choice of $\pi$, whence we may assume that $\pi$ is the resolution
    \begin{equation} \label{product decomp resolution}
    \widehat\pi_x\times \id\colon \tilde{Y_x} \times \widehat{X_i^\circ}_x \lra Y_x\times \widehat{X_i^\circ}_x, 
    \end{equation}where $\widehat\pi_x\colon\tilde Y_x \to Y_x$ is a log-resolution of singularities with exceptional divisor $E_x$.  On the one hand, there is a K\"unneth-type decomposition 
$$
\Omega_{\tilde Y_x\times \widehat{X_i^\circ}_x}^p(\log E)(-E) \cong \bigoplus_{p_1 + p_2 = p} \Omega_{\tilde Y_x}^{p_1}(\log E_x)(-E_x)\otimes \Omega_{\widehat{X_i^\circ}_x}^{p_2},
$$
 where $E_x$ is the $\widehat\pi_x$-exceptional divisor.  On the other hand, the $\supth{q}$ higher direct image sheaf of $\Omega_{\tilde Y_x}^{p_1}(\log E_x)(-E_x)\otimes \Omega_{\widehat{X_i^\circ}_x}^{p_2}$ is just $R^q(\widehat\pi_x)_*\Omega_{\tilde Y_x}^{p_1}(\log E_x)(-E_x)$ since we are taking the identity on the second factor.  Therefore,  
$$
R^q\pi_*\Omega_{\tilde X}^p(\log E)(-E)_x \cong \bigoplus_{p_1} R^q(\widehat\pi_x)_*\Omega_{\tilde Y_x}^{p_1}(\log E_x)(-E_x).
$$
  The vanishing then follows by Steenbrink vanishing applied to the lowest $\Omega_{\tilde Y_x}^{p_1}(\log E_x)(-E_x)$.
\end{example}

\subsection{Hodge Theory of the Regular Locus} Next, we need to understand how the symplectic form of a primitive symplectic variety interacts with the (compactly supported) cohomology of the regular locus.  The following lemma indicates how the symplectic form extends across the singularities of $X$.

\begin{lemma} \label{logzerosreflexive}
If\, $X$ is a symplectic variety with smooth singular locus, the sheaves $\pi_*\Omega_{\widetilde{X}}^p(\log E)(-E)$ are reflexive for every $1 \le p \le 2n$.
\end{lemma}

\begin{proof}
First assume  that $X$ has isolated singularities.  Since $X$ has rational singularities, it is enough to show that $\pi_*\Omega_{\widetilde{X}}^p(\log E)(-E) \hookrightarrow \pi_*\Omega_{\widetilde{X}}^p$ is an isomorphism for each $1 \le p \le 2n$.  For $p = 2n$, this is immediate, and for $p = 2n-1$, this holds by \cite[Theorem 1.6]{kebekus2021extending}, where we note that $R^{n-1}\pi_*\mathscr O_{\widetilde{X}} = 0$.  We may therefore  assume that $1 \le p \le 2n-2$.  

 There is an exact complex 
$$
0 \lra \Omega_{\widetilde{X}}^p(\log E)(-E) \lra \Omega_{\widetilde{X}}^p \lra \Omega_{E_{(1)}}^p \lra \Omega_{E_{(2)}}^p \lra \cdots \lra \Omega_{E_{(k)}}^p \lra \cdots, 
$$
 where $E_{(k)} = \coprod_{|J| = k}\bigcap_{J} E_j$ is the pairwise union of the $k$-fold intersections of the irreducible components of~$E$; see for example \cite[Lemma 4.1]{mustactua2020local} and \cite[Proof of Corollary 14.9]{mustata2021hodge}.  In particular, there is a short exact sequence 
$$
0 \lra \Omega_{\widetilde{X}}^p(\log E)(-E) \lra \Omega_{\widetilde{X}}^p \lra \mathscr M_p \lra 0, 
$$
 where $\mathscr M_p = \ker(\Omega_{E_{(1)}}^p \to \Omega_{E_{(2)}}^p)$.  The lemma will follow if we can show that $H^0(\mathscr M_p) = 0$.  But this will follow from the Hodge theory of the exceptional divisor $E$.  Indeed, we have a complex 
$$
0 \lra H^q(\Omega_{E_{(1)}}^p) \xrightarrow{\delta_1^{p,q}} H^q(\Omega_{E_{(2)}}^p) \xrightarrow{\delta_2^{p,q}}\cdots. 
$$
  As we have seen in Section~\ref{2.4}, the mixed Hodge pieces of the cohomology of $E$ are then computed in terms of this complex: 
$$
\gr_W^{p + q}H^{p+q+r}(E, \mathbb C)^{p,q} = \ker \delta_{r+1}^{p,q}/\im\delta_r^{p,q}.
$$
 In particular, $\gr_W^pH^p(E, \mathbb C)^{p,0} \cong H^0(\mathscr M_p)$.  But by Hodge symmetry, this is isomorphic to 
$$
\gr_W^pH^p(E, \mathbb C)^{0,p} \cong H^p(E, \mathscr O_E),
$$
 which vanishes by either  \cite[Lemma 1.2]{namikawa2000extension} or \cite[Corollary 14.9]{mustata2021hodge} for $1 \le p \le 2n-2$. 
 
 More generally, suppose that the singular locus of $X$ is smooth.  The problem is local and independent of the chosen resolution of singularities; we can then assume that $\pi$ is the resolution of singularities corresponding to a product decomposition given in Proposition~\ref{symstrat}.  In particular, the sheaves $\pi_*\Omega_{\widetilde{X}}^p(\log E)(-E)$ are reflexive if and only if the sheaves $(\widehat{\pi}_x)_*\Omega_{\tilde{Y_x}}^p(\log E_x)(-E_x)$ are reflexive for a transversal slice $Y_x$.  By assumption, $Y_x$ has at worst isolated singularities, and so the claim follows from the previous argument. 
\end{proof}

In particular, the extension of the symplectic form $\sigma$ gives a well-defined global section 
$$
\tilde\sigma \in H^0\left(\widetilde{X}, \Omega_{\widetilde{X}}^2(\log E)(-E)\right) \subset H^0\left(\widetilde{X}, \Omega_{\widetilde{X}}^2(\log E)\right),
$$
  and so we get morphisms
\begin{equation} \label{wedge morphisms}
    \tilde \sigma^p\colon \Omega_{\widetilde{X}}^{n-p}(\log E) \lra \Omega_{\widetilde{X}}^{n+p}(\log E), \quad \tilde \sigma^p\colon \Omega_{\widetilde{X}}^{n-p}(\log E)(-E) \lra \Omega_{\widetilde{X}}^{n+p}(\log E)(-E)
\end{equation} which are induced by wedging.  We emphasize that, unlike the corresponding morphisms $\sigma^p\colon \Omega_U^{n-p} \to \Omega_U^{n+p}$, the maps (\ref{wedge morphisms}) are almost never isomorphisms.  However, we note there is an important interaction between the sheaves $\Omega_{\tilde X}^p(\log E), \Omega_{\tilde X}^{2n-p}(\log E)(-E)$ and the holomorphic extension $\tilde \sigma$.  The following lemma, which is a local computation, is a consequence of Lemma~\ref{logzerosreflexive} and the canonical representation of the symplectic form on the regular locus.

\begin{lemma}\label{lemma factorization wedge}
Let $X$ be a symplectic variety of dimension $2n \ge 4$ with isolated singularities and holomorphic symplectic form $\sigma \in H^0(X, \Omega_X^{[2]})$.  For $0 \le p \le n$, there is a morphism 
$$
\mathbf R\pi_*\widetilde \sigma^p\colon \mathbf R\pi_*\Omega_{\widetilde X}^{n-p}(\log E) \lra \mathbf R\pi_*\Omega_{\widetilde X}^{n+p}(\log E)(-E)
$$
 for any log-resolution of singularities $\pi\colon\widetilde X \to X$, where $\widetilde \sigma$ is the unique extension of $\sigma$ to $H^0(\widetilde X, \Omega_{\widetilde X}^2)$.
\end{lemma} 

\begin{proof}
For any $x \in U : = X_{\mathrm{reg}}$, there exist an (analytic) neighborhood $U_x$ of $x$ and local coordinates $z_1,\ldots,z_{2n}$ such that $\sigma$, considered as a holomorphic symplectic form on $U$, can be written as 
$$
\sigma = dz_1\wedge dz_2 + \cdots + dz_{2n-1}\wedge dz_{2n}
$$
 on $U_x$. By Lemma~\ref{logzerosreflexive}, there is a \textit{unique} extension of $\sigma$ to a global section $\widetilde \sigma$ of $\Omega_{\widetilde X}^2(\log E)(-E)$ for any log-resolution of singularities $\pi\colon\widetilde X \to X$, which we describe over the exceptional divisor $E$.  For each point $x_0 \in E$, choose local coordinates $z_1',\ldots,z_{2n}'$ of $\widetilde X$ around $x_0$ such that $E = V(z_1'\cdots z_k')$ for some $k \le 2n$.  Then 
$$
\widetilde \sigma = h(dz_1'\wedge dz_2' + \cdots +dz_{2n-1}'\wedge dz_{2n}')
$$
 in a neighborhood $V_x$ of $x_0$, where $h \in H^0(V_x, \mathscr O_{V_x})$.    

 Consider the global morphism (\ref{wedge morphisms}) $\widetilde \sigma^p\colon \Omega_{\widetilde X}^{n-p}(\log E) \to \Omega_{\widetilde X}^{n+p}(\log E)$ induced by wedging.  Let $\alpha$ be a section of $\Omega_{\widetilde X}^{n-p}(\log E)(V_x)$, written as
 \begin{equation}\label{equation log n-p form}
    \alpha = f\frac{dz_{i_1}'}{z_{i_1}'}\wedge\cdots\wedge\frac{dz_{i_l}'}{dz_{i_l}}\wedge dz_{j_1}'\wedge dz_{j_1}'\wedge\cdots\wedge dz_{j_m}', \quad 1 \le i_1,\ldots,i_l \le 2r,~2r+1 \le j_1,\ldots,j_m \le 2n,
 \end{equation}
 where $f \in H^0(V_x, \mathscr O_x)$ and $l + m = n-p$, and consider the logarithmic form $\alpha \wedge \widetilde{\sigma}^p$.  Note that if $k > 2n-1$, then $\widetilde \sigma \wedge\alpha$ vanishes along $E$.  Assume without loss of generality that $k = 2r$ for $r \le n-1$.  By Lemma~\ref{logzerosreflexive}, $h$ must vanish along $z_{2r+1}',\ldots,z_{2n}'$.  It follows immediately from (\ref{equation log n-p form}) that the logarithmic form 
$$
\alpha \wedge h^p\sum_{|I| = p} dz_I',
$$
 where $I \subset \left\{1,\ldots,2n\right\}$ and $dz_I'$ is the $p$-fold wedge product of $dz_j'$ for $j \in I$, must vanish along $E$.  The logarithmic wedging map (\ref{wedge morphisms}) must factor as \[ \begin{tikzcd}
    \Omega_{\widetilde X}^{n-p}(\log E) \arrow{r}{\widetilde \sigma^p} \arrow{dr} & \Omega_{\widetilde X}^{n+p}(\log E) \\ & \Omega_{\widetilde X}^{n+p}(\log E)(-E)\rlap{.} \arrow{u}
\end{tikzcd} \] The claim follows by taking cohomology.   
\end{proof}

\subsection{Symplectic Hard Lefschetz} \label{subsection symplectic Hard Lefschetz}

Lemma~\ref{logzerosreflexive} indicates what kind of zeros the powers $\tilde \sigma^p$ of the extended symplectic form pick up across the exceptional divisor $E$.  Moreover, it says that the symplectic form on $U$ defines a class in the compactly supported cohomology, as 
\begin{equation} \label{compactly supported section}
   \tilde \sigma \in H^0\left(\widetilde{X}, \Omega_{\widetilde{X}}^p(\log E)(-E)\right) \cong \gr_W^2H^2_c(U, \mathbb C)^{2,0},   
\end{equation}
by (\ref{noncanonical log zero}).   

For convenience, we want to consider the case that $X$ has at worst terminal singularities.  By passing to a $\mathbb Q$-factorial terminalization of some bimeromorphic model, we will see that this assumption is sufficient once we prove the LLV structure theorem.

We fix a log-resolution of singularities $\pi\colon\widetilde{X} \to X$ with exceptional divisor $E$.  If $X$ has isolated singularities, the intersection cohomology groups are given by the pure Hodge structures
\[ \IH^k(X, \mathbb C) = \begin{cases} \label{IH for iso}
      H^k(U, \mathbb C), & k < 2n, \\
      \im(H^k_c(U, \mathbb C) \to H^k(U, \mathbb C)), & k = 2n, \\
      H^k_c(U, \mathbb C), & k > 2n; 
   \end{cases}
\]
see \cite[Section~6.1]{goresky1980intersection} and \cite[Corollary (1.14)]{steenbrink1976mixed}.  For $k< 2n$, the degeneration of the logarithmic Hodge-to-de Rham spectral sequence at $E_1$ induces the Hodge filtration on the pure Hodge structure $H^k(U, \mathbb C)$.  As a consequence of Theorem~\ref{theorem hodge to de rham}, we see that the intersection cohomology of a primitive symplectic variety satisfies
\[ \IH^{p,q}(X, \mathbb C) \cong \begin{cases} 
      H^q(U, \Omega_U^p), & p+q < 2n -1, \\
      H^q_c(U, \Omega_U^p), & p+ q > 2n + 1.
   \end{cases}
\] To see this, note that since $H^k(U, \mathbb Q)$ is a pure Hodge structure for $k < 2n$, the natural morphism $H^k(\widetilde{X}, \mathbb Q) \to H^k(U, \mathbb Q)$ is a surjective morphism of pure Hodge structures.  The $(p,q)$-part of the canonical Hodge structure on the cohomology of $\widetilde{X}$ must factor through $H^q(U, \Omega_U^p)$, which injects by Theorem~\ref{theorem hodge to de rham}, and so we have $H^q(U, \Omega_U^p) \xrightarrow{\lowsim} \IH^{p,q}(X)$ for $p + q < 2n-1$.  By Poincar\'e duality, we get the statement concerning the compactly supported cohomology groups. 

It turns out that this gives us enough of the intersection Hodge diamond to show that the symplectic symmetry holds. This is the key input to the construction of the LLV algebra.

\begin{theorem} \label{symplecticsymmetry}
Let $X$ be a primitive symplectic variety of dimension $2n$ with isolated singularities.  For $0 \le p \le n$ and $0 \le q \le 2n$, the cupping map 
$$
L_\sigma^p\colon \IH^{n-p,q}(X) \lra \IH^{n+p,q}(X)
$$
 is an isomorphism.
\end{theorem}

\begin{proof}
  From the above discussion, the theorem holds for $(n-p) + q < 2n$ and $(n+p) + q < 2n$ since $\sigma^p\colon H^q(U, \Omega_U^{n-p}) \to H^q(U, \Omega_U^{n+p})$ is an isomorphism.  Similarly, the theorem holds if $(n-p) + q > 2n$ and $(n+p) + q > 2n$. We check the remaining cases.

First, suppose that $(n-p) + q < 2n-1$ and $(n+p) + q > 2n +1$ (that is, we are mapping across middle cohomology).  Since $\tilde \sigma$ defines a compactly supported global section on $U \subset \widetilde{X}$, we have the factorization 

\[
\begin{tikzcd}
H^q(U, \Omega_U^{n-p}) \arrow{d} \arrow{dr}{\sigma^p} & \\
H^q_c(U, \Omega_U^{n+p}) \arrow{r} & H^q(U, \Omega_U^{n+q})\rlap{.}
\end{tikzcd}
\]
The diagonal morphism is an isomorphism since $U$ is symplectic.  Therefore, the cupping morphism will be an isomorphism if $H^q(U, \Omega_U^{n-p})$ and $H_c^q(U, \Omega_U^{n+p})$ have the same dimension.  Writing 
$$
h^{p,q}(U) = \dim H^q(U, \Omega_U^p), \quad h_c^{p,q}(U)  = \dim H^q_c(U, \Omega_U^p),
$$
 this follows since \begin{equation} \label{mixed hodge numbers symmetry}
    h^{n+p,q}_c(U) = h^{n-p,2n-q}(U) = h^{2n-q,n-p}(U) = h^{q,n-p}(U) = h^{n-p,q}(U).
\end{equation}

We now need to consider the case that we map to or from $\IH^{2n-1}(X, \mathbb C) \cong H^{2n-1}(U, \mathbb C)$, where the proof of Theorem~\ref{theorem hodge to de rham} breaks down.  First assume  that we map into $H^{2n-1}(U, \mathbb C)$. 
 We need to show that the morphism 
$$
\widetilde \sigma^p\colon H^q\left(\widetilde X, \Omega_{\widetilde X}^{n-p}(\log E)\right) \lra H^q\left(\widetilde X, \Omega_{\widetilde X}^{n+p}(\log E)\right)
$$
 is an isomorphism for $(n + p) + q = 2n-1$.  Following the proof of the degeneration of Hodge-to-de Rham, we have shown that the restriction morphisms 
$$
H^q\left(\widetilde{X}, \Omega_{\widetilde{X}}^{n+p}(\log E)\right) \lra H^q\left(U, \Omega_U^{n+p}\right)
$$
 are at least injective in this range.  Consider the commutative diagram \[ \begin{tikzcd}
H^q\left(\widetilde{X}, \Omega_{\widetilde{X}}^{n-p}(\log E)\right) \arrow{r} \arrow[swap]{d}{\tilde \sigma^{p}} & H^q\left(U, \Omega_U^{n-p}\right) \arrow{d}{\sigma^{p}} \\
H^q\left(\widetilde{X}, \Omega_{\widetilde{X}}^{n+p}(\log E)\right) \arrow{r} & H^q\left(U, \Omega_U^{n+p}\right)\rlap{.}
\end{tikzcd}
\] The top morphism is an isomorphism by Theorem~\ref{theorem hodge to de rham} unless $p = 0$, which we do not need to check, and the map $\sigma^{p}$ is an isomorphism since $U$ is symplectic.  This proves the restriction morphism is also an isomorphism for $(n + p) + q = 2n-1$, and therefore $\widetilde \sigma^p$ is an isomorphism, too.  

Now we consider the case that we map from $\IH^{2n-1}(X, \mathbb C)$, where we necessarily map across middle cohomology.  Then assume  that $(n-p) + q = 2n-1$.  Consider the commutative diagram
\[
\begin{tikzcd}
H^q\left(\widetilde{X}, \Omega_{\widetilde{X}}^{n-p}(\log E)\right) \arrow{r} \arrow[swap]{d}{\tilde \sigma^{p}} & H^q\left(U, \Omega_U^{n-p}\right) \arrow{dr} \arrow{d}{\sigma^{p}} & \\
H^q\left(\widetilde{X}, \Omega_{\widetilde{X}}^{n+p}(\log E)(-E)\right) \arrow{r} & H^q_c\left(U, \Omega_U^{n+p}\right) \arrow{r} & H^q\left(U, \Omega_U^{n+p}\right)\rlap{,}
\end{tikzcd}
\]
where the left vertical morphism follows from Lemma~\ref{lemma factorization wedge}.  Again, we are using the fact that $X$ is terminal, so that $H^q(U, \Omega_U^{n-p}) \to H^q(U, \Omega_U^{n+p})$ factors through $H_c^q(U, \Omega_U^{n+p})$.  The top morphism is injective, and so the cupping map $\tilde \sigma^{p}\colon H^q(\widetilde{X}, \Omega_{\widetilde{X}}^{n-p}(\log E)) \to H^q(\widetilde{X}, \Omega_{\widetilde{X}}^{n+p}(\log E)(-E))$ is injective.  A similar symmetry argument to (\ref{mixed hodge numbers symmetry}) shows that these groups have the same dimension.

Finally, we are left to check the case that $L_\sigma^p$ maps into middle cohomology.  These morphisms are of the form 
$$
L_\sigma^p\colon \IH^{n-p,n-p}(X) \lra \IH^{n+p,n-p}(X).
$$
 for $p < n$.  By a similar commutative diagram argument to above, we can see that these maps must be injective.  To prove surjectivity, consider a class $a \in \IH^{n+p,n-p}(X) = (W_{2n}H^{2n}(U, \mathbb C))^{n+p,n-p}$.  We note that the restriction map $H^{2n}(\widetilde{X}, \mathbb C) \to W_{2n}H^{2n}(U, \mathbb C)$ is surjective in this case; since $U$ is smooth, we may represent $a = [\alpha]$ as the class of an $(n+p,n-p)$-form $\alpha$ which is $\overline \partial$-closed.  At the level of sheaves, we have isomorphisms 
$$
\sigma^p\colon\mathcal A^{n-p,n-p}_U \lra \mathcal A^{n+p,n-p}_U
$$
sending a form $\beta$ to $\sigma^p\wedge \beta$.  If $\alpha = \sigma^p\wedge \beta$, then $0 = \overline\partial \alpha = \overline \partial (\sigma^p\wedge \beta) = \sigma^p\wedge \overline \partial \beta$.  But since $\sigma$ is symplectic, $\overline \partial \beta = 0$ and 
defines a class $[\beta] \in H^{n-p}(U, \Omega_U^{n-p})$.  This proves the surjectivity. 

To finish, note that the only remaining case is mapping from middle cohomology; but this follows from Poincar\'e duality.
\end{proof}

\section{Hard Lefschetz for Symplectic Varieties} \label{4}

In this section, we describe a hard Lefschetz theorem for the classes $\sigma$ and $\overline \sigma$ using Theorem~\ref{symplecticsymmetry}, which inherits $\IH^*(X, \mathbb C)$ with the structure of an $\mathfrak{sl}_2\times \mathfrak{sl}_2$-representation.  This data, along with the monodromy representation of the second cohomology, completely describes the $\mathfrak g$-representation structure of the intersection cohomology.

    \subsection{Symplectic Hard Lefschetz} \label{4.1}
    
    Let $L_\sigma$ be the cupping morphism with respect to the class of the symplectic form $\sigma$, and let $\mathfrak s_\sigma$ be the completed $\mathfrak{sl}_2$-triple 
$$
\mathfrak s_{\sigma} = \langle L_{\sigma}, \Lambda_{\sigma}, H_{\sigma}\rangle, 
$$
 where $H_{\sigma}$ is the weight operator of the corresponding weight decomposition.

    The isomorphisms $L_\sigma^p\colon \IH^{n-p,q}(X) \xrightarrow{\lowsim} \IH^{n+p,q}(X)$ induce a primitive decomposition on the intersection cohomology classes.  A class $\alpha \in \IH^{p,q}(X)$ for $p \le q$ is said to be \textit{$\sigma$-primitive} if $L_{\sigma}^{n-p+1}\alpha = 0$.  We denote the subspace of primitive intersection $(p,q)$-classes by $\IH_{\sigma}^{p,q}(X) \subset \IH^{p,q}(X)$.  As in the case of classic hard Lefschetz, there is a $\sigma$-primitive decomposition 
$$
\IH^{p,q}(X) = \bigoplus_{j \ge \mathrm{max}\left\{p-n,0\right\}} L_{\sigma}^j\IH_{\sigma}^{p-2j,q}(X).
$$
 While the hard Lefschetz theorem is a result of K\"ahler geometry, the relationship between the primitive decomposition theorem and the weight decomposition is purely algebraic.  In particular, we get the following. 

    \begin{corollary} \label{symweightop}
Let $X$ be a primitive symplectic variety with isolated singularities.  For $\alpha \in \IH^{p,q}(X)$, we have $H_{\sigma}(\alpha) = (p-n)\alpha$.
    \end{corollary}

    \begin{proof}
    The dual Lefschetz operator $\Lambda_{\sigma}$ can be uniquely defined to satisfy $[L_{\sigma}, \Lambda_{\sigma}] = (p-n)\alpha$ via the primitive decomposition.  In particular, if $\alpha = \sum_j L_{\sigma}^j\alpha_{\sigma}^{p-2j,k}$ is the $\sigma$-primitive decomposition of a class $\alpha \in \IH^{p,q}(X)$, then 
$$
\Lambda_{\sigma}(\alpha) := \sum_jj(n-p+j+1)L_{\sigma}^{j-1}\alpha_{\sigma}^{p-2j,q}.
$$
One checks that this operator acts on the intersection cohomology module as a degree $(-2,0)$ and satisfies the desired commutativity relation.  See \cite[Section~1.4]{kleiman1968algebraic} for more details.
    \end{proof}

    Now consider  the cupping operator $L_{\overline\sigma}\colon \IH^{p,q}(X) \to \IH^{p,q+2}(X)$.  By Hodge symmetry, it induces isomorphisms $L_{\overline\sigma}^q\colon \IH^{p,n-q}(X) \xrightarrow{\lowsim} \IH^{p,n+q}(X)$.  Considering the induced $\mathfrak{sl}_2$-triple 
$$
\mathfrak s_{\overline\sigma} = \langle L_{\overline\sigma}, \Lambda_{\overline\sigma}, H_{\overline\sigma}\rangle,
$$
 it then follows that the corresponding weight operator satisfies $H_{\overline\sigma} = (q-n)\alpha$ for $\alpha \in \IH^{p,q}(X)$. As in the holomorphic case, a class $\alpha \in \IH^{p,q}(X)$ with $q \le n$ is \textit{$\overline\sigma$-primitive} if $L_{\overline\sigma}^{n-q+1}\alpha = 0$.  If $\IH_{\overline\sigma}^{p,q}(X) \subset \IH^{p,q}(X)$ is the subspace of $\overline\sigma$-primitive $(p,q)$-classes, there is also a decomposition 
$$
\IH^{p,q}(X) = \bigoplus_{k \ge \mathrm{max}\left\{q-n,0\right\}}L_{\overline\sigma}^k\,\IH_{\overline\sigma}^{p,q-2k}(X).
$$
 If $\alpha = \sum_k L_{\overline\sigma}^k\,\alpha_{\overline\sigma}^{p,q-2k}$ is a $\overline\sigma$-primitive decomposition, then we have 
$$
\Lambda_{\overline\sigma}(\alpha) = \sum_kk(n-q+k+1)L_{\overline\sigma}^{k-1}\alpha_{\overline\sigma}^{p,q-2k}.
$$
 As in case of compact K\"ahler manifolds, a class $\alpha \in \IH^{p,q}(X)$ is $\sigma$-primitive (resp.\ $\overline\sigma$-primitive) if $\Lambda_{\sigma}(\alpha) = 0$ (resp.\ $\Lambda_{\overline\sigma}(\alpha) = 0$).

    \begin{proposition}\label{symdualcommute}
   We have $[\Lambda_{\sigma}, \Lambda_{\overline\sigma}] = 0$.
    \end{proposition}

    \begin{proof}
   Let $\alpha \in \IH^{p,q}(X)$.  We get two primitive decompositions with respect to $\sigma$ and $\overline\sigma$: 
$$
\alpha = \sum_jL_{\sigma}^j\alpha_{\sigma}^{p-2j,q} = \sum_kL_{\overline\sigma}^k\,\alpha_{\overline\sigma}^{p,q-2k},
$$
 where $\alpha_{\sigma}^{p-2j,q}$ is $\sigma$-primitive for each $j$ and $\alpha_{\overline\sigma}^{p,q-2k}$ is $\overline\sigma$-primitive for each $k$.

    Suppose that $\alpha$ is $\sigma$-primitive, so that $p \le n$ and $L_{\sigma}^{n-p+1}\alpha = 0$.  Then 
$$
0 = L_{\sigma}^{n-p+1}\alpha = \sum_j L_{\sigma}^{n-p+1}L_{\overline\sigma}^k\,\alpha_{\overline\sigma}^{p,q-2k} = \sum_k L_{\overline\sigma}^k\left(L_{\sigma}^{n-p+1}\alpha_{\overline\sigma}^{p,q-2k}\right).
$$
  If $\alpha_{\overline\sigma}^{p, q-2k}$ is $\overline\sigma$-primitive of degree $q-2k$, then $L_{\sigma}^{n-p+1}\alpha_{\overline\sigma}^{p,q-2k}$ is also $\overline\sigma$-primitive of degree $q-2k$, as the antiholomorphic degree is not changing and the cup products commute.  Therefore,  
$$
\sum_kL_{\overline\sigma}^k\left(L_{\sigma}^{n-p+1}\alpha_{\overline\sigma}^{p,q-2k}\right)
$$
 is a $\overline\sigma$-primitive decomposition of $L^{n-p+1}\alpha = 0$, and so $L_{\overline\sigma}^k\,L_{\sigma}^{n-p+1}\alpha_{\overline\sigma}^{p,q-2k} = 0$ for each $k$.  This implies that $L^{n-p+1}_{\sigma}\alpha_{\overline\sigma}^{p,q-2k} = 0$, as $L^k_{\overline\sigma}$ is injective over antiholomorphic degree $q-2k$, and so $\alpha_{\overline\sigma}^{p,q-2k}$ is both $\overline\sigma$- and $\sigma$-primitive for each $k$.  Similarly, $\alpha_{\sigma}^{p-2j,q}$ is both $\sigma$- and $\overline\sigma$-primitive for each $j$ if $\alpha$ is $\overline\sigma$-primitive.

    We have shown that any $(p,q)$-class $\alpha \in \IH^{p,q}(X)$ admits a simultaneous $\sigma$- and $\overline{\sigma}$-primitive decomposition. This implies that the operators $\Lambda_\sigma$ and $\Lambda_{\overline \sigma}$ commute.
    \end{proof}
    
    \begin{remark} \label{Torus remark}
    We note that the symplectic hard Lefschetz theory also holds for even-dimensional complex tori, as they admit a symplectic form.  Looijenga--Lunts observed that the total Lie algebra of a complex torus detects the Hodge theory, see \cite[Section~3]{looijenga1997lie}, but the LLV algebra has a different structure than a compact hyperk\"ahler manifold.  This is because the Torelli theorem for a complex torus does not differentiate between it and its dual.
    \end{remark}
    
    \subsection{Hard Lefschetz for Non-Isotropic Classes} \label{4.2}
    
    For the rest of the section, we will assume that $X$ is a primitive symplectic variety with isolated singularities and $b_2 \ge 5$.  In particular, the global moduli theory of Bakker--Lehn \cite{bakker2018global} applies.  In \cite{tighe2023thesis}, the author shows that $\IH^{1,1}(X)$ parametrizes isomorphism classes of certain deformations of $X$, where we allow smoothing of the singularities; these deformations are unobstructed, and we get a global Torelli theorem in terms of the full intersection cohomology.  The results of Section~\ref{2} will allow us to avoid developing this theory here.

    We now turn to classic hard Lefschetz.  Up to some locally trivial deformation, the classes $\sigma, \overline \sigma \in H^2(X, \mathbb C)$ define real cohomology classes $\gamma = \mathfrak{R}(\sigma)$ and $\gamma' = \mathfrak{I}(\sigma)$, corresponding to the real and imaginary parts of $\sigma$: 
$$
\gamma = \sigma + \overline \sigma, \quad \gamma' = -i(\sigma - \overline \sigma).
$$
 There are cupping maps $L_\gamma, L_{\gamma'}$ on the intersection cohomology module by the proof of Proposition~\ref{HLregclass}; these morphisms, considered as nilpotent operators in $\mathfrak{gl}(\IH^*(X, \mathbb Q))$, complete to $\mathfrak{sl}_2$-triples 
$$
\mathfrak{s}_\gamma = \langle L_\gamma, \Lambda_\gamma, H_\gamma\rangle, \quad \mathfrak{s}_{\gamma '} = \langle L_{\gamma '}, \Lambda_{\gamma'}, H_{\gamma'}\rangle.
$$

    \begin{proposition} \label{realimaginaryHL}
  The cohomology classes  $\gamma$ and $\gamma'$ are HL.
    \end{proposition}

    \begin{proof}
    By assumption $q_X(\gamma),q_X(\gamma') \ne 0$.  Therefore, this follows from Proposition~\ref{HLregclass}.
    \end{proof}

    In particular, the weight operators satisfy $H_\gamma = H_{\gamma'} = H$, where $H$ acts as $(k-2n)\id$ on $\IH^k(X, \mathbb R)$. 

    \begin{corollary} \label{realimaginarycommute}
   We have $[\Lambda_\gamma, \Lambda_{\gamma'}] = 0$.
    \end{corollary}

    \begin{proof}
    Note that $\Lambda_\gamma$ and $\Lambda_{\gamma'}$ act as 
$$
\Lambda_\gamma = \Lambda_\sigma + \Lambda_{\overline{\sigma}}, \quad \Lambda_{\gamma'} = i(\Lambda_{\sigma} - \Lambda_{\overline \sigma}),
$$
 as they are defined by the commutator relation $H = [L_\gamma, \Lambda_\gamma] = [L_{\gamma'}, \Lambda_{\gamma'}]$. Taking the commutator, the result follows since $[\Lambda_\sigma, \Lambda_{\overline \sigma}] = 0$ by Proposition~\ref{symdualcommute}.
    \end{proof}
    
    We therefore have a pair of non-isotropic classes $(\gamma,\gamma')$ such that $q_X(\gamma,\gamma') = 0$ whose dual Lefschetz operators commute.  As we will see, it is important that this pair vanishes under the induced bilinear form, as this defines a Zariski-closed condition on the space of non-isotropic pairs.

    \begin{corollary} \label{commutecorollary}
   We have $[L_\sigma, \Lambda_{\overline{\sigma}}] = -[L_{\overline \sigma}, \Lambda_{\sigma}] = 0$.
    \end{corollary}

    \begin{proof}
    Note that 
$$
[L_\gamma, \Lambda_\gamma] = [L_\sigma, \Lambda_\sigma] + [L_\sigma, \Lambda_{\overline \sigma}] + [L_{\overline \sigma}, \Lambda_\sigma] + [L_{\overline \sigma}, \Lambda_{\overline \sigma}]
$$
 and 
$$
[L_{\gamma'}, \Lambda_{\gamma'}] = [L_\sigma, \Lambda_\sigma] - [L_\sigma, \Lambda_{\overline \sigma}] - [L_{\overline \sigma}, \Lambda_\sigma] + [L_{\overline \sigma}, \Lambda_{\overline \sigma}].
$$
  Either equation implies that $[L_\sigma, \Lambda_{\overline \sigma}] = -[L_{\overline \sigma}, \Lambda_{\sigma}]$ since $([L_\sigma, \Lambda_{\sigma}] + [L_{\overline \sigma}, \Lambda_{\overline\sigma}])(\alpha) = (p+q-2n)\alpha$ for $\alpha \in \IH^{p,q}$ (see Corollary~\ref{symweightop} and the following paragraph).  Since $[L_\gamma, \Lambda_\gamma] = [L_{\gamma'}, \Lambda_{\gamma'}]$, subtracting the second equation from the first implies $[L_\sigma, \Lambda_{\overline \sigma}] = [L_{\overline \sigma}, \Lambda_\sigma] = 0$.
    \end{proof} 

    \begin{corollary} \label{intersectionVerbitskyresult}
    For $\alpha \in \IH^{p,q}(X)$, we have $[L_{\gamma}, \Lambda_{\gamma'}] = i(p-q)\alpha$.
    \end{corollary}

    \begin{proof}
    Note that $-i[L_\gamma, \Lambda_{\gamma'}] = [L_\sigma, \Lambda_\sigma] - [L_\sigma,\Lambda_{\overline \sigma}] + [L_{\overline \sigma}, \Lambda_{\sigma}] - [L_{\overline \sigma}, \Lambda_{\overline \sigma}]$.  The result follows since the two middle  terms vanish by Corollary~\ref{commutecorollary}.
    \end{proof}
    
    Corollary~\ref{intersectionVerbitskyresult} is a generalization of Verbitsky's generalization of hard Lefschetz on compact hyperk\"ahler manifolds, see \cite{verbitsky1990action}: If $(X,g,I,J,K)$ is a compact hyperk\"ahler manifold with hyperk\"ahler metric $g$ and complex structures $I,J,K$, the Weil operators $C_\sigma = [L_{\omega_\lambda}, \Lambda_{\omega_\lambda}] = i(p-q)\id$ are contained in the algebras $\mathfrak g_{g} \cong \mathfrak{so}(4,1)$ (see Theorem~\ref{so(4,1)}).  

\section{The LLV Algebra for Intersection Cohomology} \label{5}

In this section, we prove the LLV structure theorem for intersection cohomology.  We  do this as follows: First, we prove the theorem for $\mathbb Q$-factorial singularities, which will allow us to consider all HL elements of $H^2(X, \mathbb C)$ and use the monodromy density theorem under the assumption $b_2 \ge 5$.  We then show that the theorem holds in general by passing to a $\mathbb Q$-factorial terminalization, using the semismallness of such crepant morphisms; see Corollary~\ref{Qfactorialterminalsemismall}.

The key here is to observe that we have a pair $(\gamma,\gamma')$ of non-isotropic vectors such that $q_X(\gamma,\gamma') = 0$.  As we will see below, the second condition is necessary.  In \cite[Theorem 2.2]{huybrechts2001infinitesimal}, it was shown that the commutativity of dual Lefschetz operators holds infinitesimally, and so we always have a pair of non-isotropic pairs which commute.  The point though is that we get a variety $L$, described by \textit{Zariski-closed conditions}, for which the monodromy group acts. 

\subsection{$\mathbb Q$-Factorial Terminal Case} As in the smooth case, we define the LLV algebra in terms of HL classes in intersection cohomology.

    \begin{definition}
        If $X$ is a primitive symplectic variety, the (\textit{intersection}) \textit{LLV algebra} is 
$$
\mathfrak g : = \mathfrak g_{\mathrm{tot}}(X) = \langle L_\alpha, \Lambda_\alpha|~\alpha \in \IH^2(X, \mathbb Q)~\mathrm{is~HL}\rangle.
$$
        \end{definition}
    
    Now assume  that $X$ is a primitive symplectic variety with at worst $\mathbb Q$-factorial terminal singularities.  By Proposition~\ref{Qcrit}, the LLV algebra is generated by the dual Lefschetz operators corresponding to HL classes in $H^2(X, \mathbb Q)$.
    
        \subsubsection{Commutativity of the dual Lefschetz operators} Recall that by Corollary~\ref{realimaginarycommute}, there exists a point $(\gamma,\gamma') \in H^2(X, \mathbb R)^{\times 2}$ of non-isotropic vectors such that $[\Lambda_{\gamma}, \Lambda_{\gamma'}] = 0$.  Moreover, we have $q_X(\gamma,\gamma') = 0$, where $q_X(-,-)$ is the induced bilinear form with respect to $q_X$.  The upshot is that the monodromy group acts on the space of such non-isotropic pairs.
        
        \begin{theorem}\label{Qfactorialcommute}
        Let $X$ be a $\mathbb Q$-factorial primitive symplectic variety with isolated singularities and $b_2 \ge 5$.  For any pair of non-isotropic classes $\alpha, \beta \in H^2(X, \mathbb Q)$, we have $[\Lambda_\alpha, \Lambda_\beta] = 0$.
        \end{theorem}
        
        \begin{proof}
        Let $X$ be a primitive symplectic variety with $\mathbb Q$-factorial terminal isolated singularities, so that Theorem~\ref{symplecticsymmetry} and the Lefschetz theory of Section~\ref{4.1} hold. If $\sigma$ is the class of the symplectic form, let $\gamma = \mathfrak R(\sigma)$ and $\gamma' = \mathfrak I (\sigma)$.  We have seen that $[\Lambda_\gamma, \Lambda_{\gamma'}] = 0$, which is a Zariski-closed condition.  Consider the space 
$$
L = \left\{(\alpha, \beta) \in H^2(X, \mathbb C) \times H^2(X, \mathbb C)\mid q_X(\alpha),q_X(\beta) \ne 0,~~q_X(\alpha,\beta) = 0\right\}.
$$
The group $\SO(\Gamma_{\mathbb C})$ acts on $L$ diagonally.  If $\Mon(X) \subset O(\Gamma)$ is the monodromy group associated to $X$ (see Remark~\ref{monodromyremark}), let $G_X : = \SO(\Gamma) \cap \Mon(X) \subset \SO(\Gamma_{\mathbb C})$.
Note that the $(\gamma, \gamma')$-orbit of $G_X$ preserves the commutator relation: If $g \in G_X$, then $[\Lambda_{g\cdot \gamma}, \Lambda_{g\cdot \gamma'}] = 0$.  Indeed, the dual Lefschetz operators satisfy the property 
$$
\Lambda_{g\cdot \gamma} = g\Lambda_{\gamma}g^{-1}
$$
 as they are uniquely determined by $[L_{g\cdot \gamma}, \Lambda_{g\cdot \gamma}] = (k-2n)\id$, and the monodromy group is invariant under conjugation.  By the Borel density theorem (see \cite{furstenberg1976note}), $G_X$ is Zariski dense in $\SO(\Gamma_{\mathbb C})$, and $[\Lambda_{g\cdot \gamma}, \Lambda_{g\cdot \gamma'}] = 0$ for every $g \in \SO(\Gamma_{\mathbb C})$.  The commutativity then holds for every $(\alpha, \beta) \in L$. 
        \end{proof}

        Thus, the LLV algebra $\mathfrak g$ only contains elements of degrees $2,0,-2$.  As in Section~\ref{2.7}, we define $\mathfrak g_2$ as the subalgebra generated by the $L_\alpha$, $\mathfrak g_{-2}$ as the subalgebra generated by the $\Lambda_\alpha$, and $\mathfrak g_0$ as the degree 0 piece which is (necessarily) of the form 
$$
\mathfrak g_0 \cong \overline{\mathfrak g}\times \mathbb Q\cdot H.
$$

        \subsubsection{$\overline{\mathfrak g}$ acts via derivations} We want to show that $\mathfrak g \cong \mathfrak g_2\oplus \mathfrak g_0 \oplus \mathfrak g_{-2}$.  As in the smooth case, this will be done if we can show that $\overline{\mathfrak g}$ acts via derivations on intersection cohomology.
        
        \begin{proposition} \label{semisimplederivations}
        The semisimple part $\overline{\mathfrak g}$ acts via derivations.
        \end{proposition}
        
        \begin{proof}
        As $\overline{\mathfrak g}$ is generated by the commutators $[L_\alpha, \Lambda_\beta]$, it is enough to show that these elements act on the intersection cohomology module via derivations.  Consider again the point $(\gamma, \gamma') \in L$ as in the proof of Theorem~\ref{Qfactorialcommute}.  The commutator satisfies $[L_\gamma, \Lambda_{\gamma'}](x) = i(p-q)x$ for a $(p,q)$-class $x$.  While this identity is not preserved by the restricted monodromy group $G_X$, we note that $[L_\gamma, \Lambda_{\gamma'}]$ acts on $\IH^*(X, \mathbb R)$ via derivations, and this property is preserved via $G_X$.  Specifically, for any $g \in G_X$, $[L_{g\cdot \gamma}, \Lambda_{g \cdot \gamma'}]$ also acts via derivations.  But by the monodromy density theorem, this must be true for any $[L_\alpha, \Lambda_\beta]$, too.
        \end{proof}
        Thus, we get the eigenvalue decomposition
        \begin{equation}\label{LLVdecomposition}
            \mathfrak g \cong \mathfrak g_2\oplus \mathfrak g_0\oplus \mathfrak g_{-2}.
        \end{equation}
        
        \subsubsection{LLV structure theorem} We can now prove the LLV structure theorem for $\mathbb Q$-factorial singularities.
        
        \begin{theorem}\label{QfactorialLLV}
        Let $X$ be a $\mathbb Q$-factorial primitive symplectic variety with isolated singularities with LLV algebra $\mathfrak g$ and $b_2 \ge 5$.  Then 
$$
\mathfrak g \cong \mathfrak{so}((H^2(X, \mathbb Q),q_X)\oplus \mathfrak h), 
$$
 where $q_X$ is the BBF form and $\mathfrak h$ is a hyperbolic plane.
        \end{theorem}

        Another way of stating Theorem~\ref{QfactorialLLV} is in terms of the Mukai completion.  If we write 
$$
(H^2(X, \mathbb Q),q_X)\oplus \mathfrak h =: \left(\widetilde H^2(X, \mathbb Q), \widetilde q_X\right),
$$
 the Mukai completion $(\widetilde H^2(X, \mathbb Q), \widetilde q_X)$ inherits a bilinear structure and Hodge structure which is compatible with $(H^2(X, \mathbb Q), q_X)$.  Theorem~\ref{QfactorialLLV} then gives an isomorphism 
$$
\mathfrak g \cong \mathfrak{so}\left(\widetilde H^2(X, \mathbb Q), \widetilde q_X\right).
$$
 In the smooth case, the subalgebra $\overline{\mathfrak g} : = \mathfrak{so}(H^2(X, \mathbb Q), q_X)$ corresponds to the semisimple part of the decomposition (\ref{LLVdecomposition}).  Moreover, the completion $(H^2(X, \mathbb Q), q_X)\hookrightarrow (\widetilde H^2(X, \mathbb Q), \widetilde q_X)$ is compatible with the extension $\overline{\mathfrak g} \hookrightarrow \mathfrak g$ by branching (see \cite[Section~B.2.1]{green2019llv}).  It is then sufficient to prove the structure theorem for the semisimple part.
        
        \begin{proof}
        We will show that \begin{equation}\label{equation semisimple isomorphism}
         \overline{\mathfrak g} \cong \mathfrak{so}(H^2(X, \mathbb Q),q_X).   
        \end{equation} Since $\overline{\mathfrak g}$ acts by derivations and preserves cup product, we see that it preserves the BBF form $q_X$.  Therefore, there is an \textit{injective} map $\overline{\mathfrak g} \hookrightarrow \mathfrak{so}(H^2(X, \mathbb C),q_X)$.  The surjectivity is a reproduction of Verbitsky's result.  Consider once more the point $(\gamma,\gamma') \in L$.  By global (or local) Torelli, this point completes to a positive three-space $W_\sigma = \langle \alpha, \gamma, \gamma'\rangle$ for which $q_X$ is positive-definite.  Now Theorem~\ref{Qfactorialcommute} shows that the dual Lefschetz operators corresponding to $W_\sigma$ all commute.  Moreover, we can extend Corollary~\ref{intersectionVerbitskyresult} to the entire three-space and show that the commutators $[L_\alpha, \Lambda_\gamma]$ and $[L_\alpha, \Lambda_{\gamma'}]$ acts as the Weil operators with respect to the induced complex structure, after possibly passing to a locally trivial deformation.  It then follows  that 
$$
\mathfrak g_\sigma: = \langle L_\omega, \Lambda_\omega:~\omega \in W_\sigma\rangle
$$
 satisfies Verbitsky's commutator relations (\ref{verbitsky commutator relations}).  The upshot is that $\mathfrak g_{\sigma} \cong \mathfrak{so}(4,1)$.  Using monodromy density once more, we then see  that \textit{any} HL class $x \in H^2(X, \mathbb R)$ completes to a positive three-space $W_x$, and the corresponding algebra satisfies $\mathfrak g_x \cong \mathfrak {so}(4,1)$.  The surjectivity then follows  as the Lie algebra  $\mathfrak{so}(H^2(X, \mathbb C), q_X)$ is generated by the simultaneous $\mathfrak{so}(4,1)$-actions.
        \end{proof}
    
    \subsection{General Case} \label{5.2}
    
    We now move to a general primitive symplectic variety.  In order for the LLV structure theorem to make sense, we need to define a BBF form for intersection cohomology.  The result will then follow from the representation theory of $\mathfrak{so}(m)$.
    
        \subsubsection{A BBF form on intersection cohomology} \label{subsubsection bbf form}
        
        \begin{definition} \label{IntBBF1}
        Let $X$ be a primitive symplectic variety of dimension $2n$ and $\pi\colon\widetilde{X} \to X$ a resolution of singularities.  For $\alpha \in \IH^2(X, \mathbb C)$, let $\tilde \alpha$ be the class in $H^2(\widetilde{X}, \mathbb C)$ under the injection $\IH^2(X, \mathbb C) \hookrightarrow H^2(\widetilde{X}, \mathbb C)$.  For $\sigma \in \IH^{2,0}(X)$, we define a quadratic form on $\IH^2(X, \mathbb C)$ by 
$$
Q_{X,\sigma}(\alpha) : = \frac{n}{2}\int_{\widetilde{X}}(\sigma\overline\sigma)^{n-1}\tilde{\alpha}^2 + (1-n)\int_{\widetilde{X}}\sigma^n\overline{\sigma}^{n-1}\tilde{\alpha}\int_{\widetilde{X}}\sigma^{n-1}\overline{\sigma}^n\tilde{\alpha}.
$$

        \end{definition}

        The definition extends the quadratic form $q_{X,\sigma}$ (see Definition~\ref{BBFdef}) and is compatible with the decomposition theorem, Theorem~\ref{intprops}.  In particular, we have the following. 

        \begin{lemma}\label{IntBBFrest}
        Let $X$ be a primitive symplectic variety of dimension $2n$.  If\, $\pi\colon\widetilde{X} \to X$ is a resolution of singularities and $\phi\colon Z \to X$ is a $\mathbb Q$-factorial terminalization, then 
$$
Q_{X,\sigma}|_{H^2(X, \mathbb C)} = q_{X,\sigma}, \quad q_{Z, \sigma_Z}|_{\IH^2(X, \mathbb C)} = Q_{X,\sigma}.
$$
 
        \end{lemma}  

        \begin{proof}
        This follows from Schwald's description of the form $q_{X, \sigma}$, see \cite{schwald2017fujiki}, where he showed that $q_{X,\sigma}(\alpha)$ agrees with 
$$
\frac{n}{2}\int_{\widetilde{X}}(\sigma\overline\sigma)^{n-1}\tilde{\alpha}^2 + (1-n)\int_{\widetilde{X}}\sigma^n\overline{\sigma}^{n-1}\tilde{\alpha}\int_{\widetilde{X}}\sigma^{n-1}\overline{\sigma}^n\tilde{\alpha}
$$
 for $\alpha \in H^2(X, \mathbb C)$.  The claim follows since we have inclusions $H^2(X, \mathbb C) \subset \IH^2(X, \mathbb C) \subset H^2(Z, \mathbb C) \subset H^2(\widetilde{X}, \mathbb C)$.
        \end{proof}

        Then note that $Q_{X,\sigma}$ is independent of $\sigma$ whenever $q_{X,\sigma}$ is.  This leads to the following definition.

        \begin{definition} \label{IntBBF2}
        The \textit{intersection Beauville--Bogomolov--Fujiki form} on $\IH^2(X, \mathbb C)$ of a primitive symplectic variety $X$ with $\int_X(\sigma\overline\sigma)^n = 1$ is the quadratic form $Q_X: = Q_{X,\sigma}$.
        \end{definition}

        \begin{proposition} \label{IntFujiki}
        Let $X$ be a primitive symplectic variety with a resolution of singularities $\pi\colon\widetilde{X} \to X$.
        \begin{enumerate}
            \item There is a constant $c \in \mathbb R_{> 0}$ such that for every $\alpha \in H^2(X, \mathbb C)$, 
$$
c\cdot Q_X(\alpha)^n = \int_{\widetilde X} \widetilde{\alpha}^n,
$$
 where $\widetilde \alpha$ is the extension of $\alpha$ under the inclusion $H^2(X, \mathbb C) \subset H^2(\widetilde X, \mathbb C)$.
    
            \item The quadratic form $Q_X$ defines a real quadratic form on $\IH^2(X, \mathbb R)$ of signature $(3,B_2 - 3)$, where $B_2 = \dim \IH^2(X, \mathbb R)$. 
        \end{enumerate}
        \end{proposition}

        \begin{proof}
        We may assume without loss of generality that $X$ is projective, so let $\phi\colon Z \to X$ be a $\mathbb Q$-factorial terminalization.  The first claim follows from Lemma~\ref{IntBBFrest} and the fact that the statement holds for $q_{Z}$ by \cite[Theorem 2(1)]{schwald2017fujiki}.
        
        For the second claim, note that by \cite[Proposition 25]{schwald2017fujiki}, there are decompositions $H^2(X, \mathbb R) = V_X^+\oplus V_X^-$ and $H^2(Z, \mathbb R) = V_{Z}^+\oplus V_{Z}^-$ which are $q_X$- and $q_{Z}$-orthogonal, respectively, such that the restrictions $q_X|_{V_X^+}$ and $q_Z|_{V_Z^+}$ (resp.\ $q_X|_{V_X^-}$ and $q_Z|_{V_Z^-}$) are positive-definite (resp.\ negative-definite) for $q_X$ and $q_Z$.  For $\sigma \in H^{2,0}(X) \cong H^{2,0}(Z)$, the positive parts can be described by $V_X^+ = V_{Z}^+ = \langle \mathrm{Re}(\sigma), \im(\sigma), \alpha\rangle$ for some ample class $\alpha \in H^2(X, \mathbb R) \subset H^2(Z, \mathbb R)$.  We can extend $V_X^+$ to $\IH^2(X, \mathbb R)$ to a 3-dimension real space $(V_X^+)' \subset \IH^2(X)$ for which $Q_X$ is positive-definite.  It is now immediate that the signature must be $(3, B_2-3)$, as having an isotropic vector subspace in $(\IH^2(X, \mathbb R), Q_X)$ would necessarily lower the rank of $(V_X^+)'$. One can define a $Q_X$-orthogonal decomposition $\IH^2(X, \mathbb R) = (V_X^+)'\oplus (V_X^-)'$ as in \cite{schwald2017fujiki}.  The negative-definite part is given by
$$
(V_X^-)' = \alpha^{\perp}\cap \IH^{1,1}(X)\cap \IH^2(X, \mathbb R), 
$$
 and the proof, which is entirely linear algebraic on the Hodge structure, is the same as for the case of $H^2(X, \mathbb C)$.
        \end{proof} 
        
        \subsubsection{Proof of Theorem~\ref{LLVint}}
        
        We now want to prove the LLV structure theorem for intersection cohomology.  The LLV algebra is defined as in the smooth case: Recall that a class $\alpha \in \IH^2(X, \mathbb Q)$ is HL if it satisfies hard Lefschetz.

        By Proposition~\ref{IntBBF2}, the BBF form on $\IH^2(X, \mathbb R)$ descends to $\IH^2(X, \mathbb Q)$, and we get a rational quadratic vector space $(\IH^2(X, \mathbb Q), Q_X)$.
        
        \begin{theorem} \label{GenLLVthm}
        If\, $X$ is a primitive symplectic variety with isolated singularities and $b_2 \ge 5$, the intersection LLV algebra satisfies 
$$
\mathfrak{g} \cong \mathfrak{so}((\IH^2(X, \mathbb Q), Q_X)\oplus \mathfrak h).
$$

        \end{theorem}

        \begin{proof}
        First suppose  that $X$ is projective.  Since $X$ has canonical singularities, it admits a $\mathbb Q$-factorial terminalization $\phi\colon Z \to X$ by \cite{birkar2010existence}.  We denote by $\mathfrak g_Z$ the LLV algebra on $Z$.  Since $\phi$ satisfies reflexive pullback of differentials, see (\ref{reflexive pullback}), $Z$ has at worst isolated singularities. By Theorem~\ref{QfactorialLLV}, 
$$
\mathfrak g_Z \cong \langle L_\beta, \Lambda_\beta|~q_Z(\beta) \ne 0\rangle \cong \mathfrak{so}((H^2(X, \mathbb Q), q_Z)\oplus \mathfrak h).
$$
 Suppose $\alpha \in \IH^2(X, \mathbb C)$ is any class which satisfies hard Lefschetz.  Since $\phi$ is semismall, then the pullback $\phi^*\alpha$, which exists since $X$ has rational singularities, must also satisfy hard Lefschetz.  Note that this means $q_Z(\phi^*\alpha) \ne 0$ by Proposition~\ref{HLregclass}, and so $Q_X(\alpha) \ne 0$ by  Lemma~\ref{IntBBFrest}.  It follows that 
$$
\mathfrak g = \langle L_\alpha, \Lambda_\alpha|~Q_X(\alpha) \ne 0\rangle.
$$
  By \cite[Remark 4.4.3]{de2005hodge}, it follows that $\Lambda_\alpha$ is a direct summand of $\Lambda_{\phi^*\alpha}$ as $\Lambda_\alpha$ is uniquely determined by the commutator relation $[L_\alpha, \Lambda_\alpha] = (k-2n)\id$.  By definition, then $\mathfrak g \subset \mathfrak g_Z$, which is a \textit{canonical} injection, and the $L_\alpha, \Lambda_\alpha$ satisfy all the same commutator relations as the $L_{\phi^*\alpha}, \Lambda_{\phi^*\alpha}$.  In particular, $\mathfrak g$ is just the restriction of $\mathfrak g_Z \cong \mathfrak{so}((H^2(Z, \mathbb Q), q_Z)\oplus \mathfrak h)$ to the subspace generated by the hard Lefschetz operators corresponding to the elements of $\IH^2(X, \mathbb C)$.  Since $q_Z|_{\IH^2(X, \mathbb Q)} = Q_X$ by Lemma~\ref{IntBBFrest}, it follows that $\mathfrak g\cong \mathfrak{so}((\IH^2(X, \mathbb Q),Q_X)\oplus \mathfrak h)$.

        In general, the total Lie algebra is a locally trivial diffeomorphism invariant, and so the Lie algebra structure of $\mathfrak g$ is preserved under small deformations.  Since a general locally trivial deformation of $X$ is projective, see \cite[Corollary 6.11]{bakker2018global}, the result follows.
        \end{proof}
        
        The proof indicates that any primitive symplectic variety which admits a $\mathbb Q$-factorial terminalization with isolated singularities also satisfies the LLV structure theorem.  This, for example, holds in dimension 4 by Proposition~\ref{symterm}.
        
        \begin{corollary}
          If\, $X$ is a primitive symplectic variety of dimension 4 and $b_2 \ge 5$, then there is an isomorphism 
          $\mathfrak g \cong \mathfrak{so}((\IH^2(X, \mathbb Q),Q_X)\oplus \mathfrak h)$.
        \end{corollary}
        
        Another interesting consequence of the proof of Theorem~\ref{GenLLVthm} is in the case $b_2 = 4$, where \cite[Theorem 1.1]{bakker2018global} does not apply.  If such an $X$ is \textit{non-terminal}, then the inclusion $\IH^2(X, \mathbb Q) \hookrightarrow H^2(Z, \mathbb Q)$ of a $\mathbb Q$-factorialization must be strict, and so $\dim H^2(Z,\mathbb Q) \ge 5$.  Therefore, $\mathfrak g_Z \cong \mathfrak{so}((H^2(Z, \mathbb Q),q_Z)\oplus \mathfrak h)$, and  $\alpha \in H^2(Z, \mathbb Q)$ is HL if and only if $q_Z \ne 0$.  We get the following.  
        
        \begin{corollary}
        If\, $X$ is a non-terminal projective primitive symplectic variety and $b_2 = 4$, then 
$$
\mathfrak g  = \langle L_\alpha, \Lambda_\alpha|~Q_X(\alpha) \ne 0 \rangle \cong \mathfrak{so}((\IH^2(X, \mathbb Q), Q_X)\oplus \mathfrak h).
$$

        \end{corollary}

        \begin{remark} \label{smoothLLVproof}
        As indicated in the introduction, we note that our proof gives an algebraic proof of the LLV structure theorem for compact hyperk\"ahler manifolds with $b_2 \ge 5$.  We note that by Verbitsky's global Torelli theorem (which depends on the existence of twistor deformations), the monodromy groups $\Mon(X) \subset O(\Gamma)$ are always finite-index subgroups, and so $G_X \subset \SO(\Gamma_{\mathbb C})$ is Zariski dense (see \cite[Remark 8.12]{bakker2018global} for more details).  The proof of Theorem~\ref{Qfactorialcommute} follows through for any $b_2$ in the smooth case.  
        \end{remark}

\subsection{Holomorphic Symplectic Orbifolds} The methods of this paper show that the LLV structure theorem can be extended to any primitive symplectic variety $X$ for which the Hodge filtration on (intersection) cohomology is generated by the symplectic form.  One other case where this can be seen is when $X$ has at worst quotient singularities.  If $X$ is a primitive symplectic orbifold, then we have the following:  
\begin{enumerate}
    \item $\IH^*(X, \mathbb Q) \cong H^*(X, \mathbb Q)$ since $X$ is a $\mathbb Q$-homology manifold; see \cite[Proposition 8.2.21]{hotta2007d}.  In particular, $H^*(X, \mathbb Q)$ satisfies hard Lefschetz, and we may define the total Lie algebra $\mathfrak g$.

    \item The Hodge filtration on $H^*(X, \mathbb C)$ is induced by a spectral sequence 
$$
E_1^{p,q} : = H^q\left(X, \Omega_X^{[p]}\right) \Longrightarrow H^{p+q}(X, \mathbb C),
$$
 which is known classically by Steenbrink \cite[Section~1.6]{steenbrink1976mixed}, but see also \cite[Corollary 4.3]{shen2023k} and \cite[Theorem 7.2]{tighe2023holomorphic} for an argument using du Bois complexes. 
\end{enumerate} We therefore get the symmetry 
$$
H^{n-p,q}(X) \overset{\lowsim}\lra H^{n+p,q}(X)
$$
 induced by the symplectic form, and the Hodge filtration is generated by the class of the symplectic form.  We conclude the following from the methods of Section~\ref{5}. 

\begin{proposition} \label{proposition LLV for orbifolds}
If\, $X$ is a primitive symplectic orbifold, then there is an isomorphism 
$$
\mathfrak g \cong \mathfrak{so}((H^2(X, \mathbb Q), q_X)\oplus \mathfrak h).
$$

\end{proposition}

\begin{proof}
From the discussion above, the only thing to remark is that we can drop the assumption that $b_2 \ge 5$ (see also Section~\ref{subsection remark on b_2}).  In this case, we can use Menet's global Torelli theorem for holomorphic symplectic orbifolds \cite[Theorem 1.1]{menet2020global} and note that the surjectivity of the period map is sufficient to prove Theorem~\ref{monodensity} (see \cite[Remark 8.12]{bakker2018global}), and so Theorem~\ref{Qfactorialcommute}, Proposition~\ref{semisimplederivations}, and Theorem~\ref{QfactorialLLV} extend to this setting.
\end{proof}

 \subsection{A Remark on $\boldsymbol{b_2 <5}$} \label{subsection remark on b_2} Our methods leave open the case $b_2 = 3$, and $b_2 =4$ in the terminal case.  It is not known if there exists a compact hyperk\"ahler manifold with $b_2 =3$ or $4$, but Verbitsky's global Torelli theorem predicts that the cohomology is completely described by its structure as an $\mathfrak{so}(b_2-2,4)$-representation.  When $b_2 = 3$, this is exactly the action of $\mathfrak{so}(4,1)$ coming from the twistor deformation. 
    
    In the singular world, there is a primitive symplectic variety with $b_2 = 3$.  It is obtained by taking the Fano variety of lines of a special cubic 4-fold admitting an automorphism of order 11; see \cite{mongardi2013symplectic}.  The induced automorphism is symplectic, and the resulting quotient has the desired Betti number; see \cite[Section~5.2]{fu2020betti}.  In this case, the action of $\mathfrak{so}(4,1)$ exists since symplectic orbifolds admit twistor deformations; see \cite[Theorem 5.4]{menet2020global}. 
    
    It is unknown if there exists a primitive symplectic variety with $b_2 = 4$.
    
    Of course, $\IH^*(X, \mathbb C)$ inherits the structure of an $\mathfrak{sl}_2$-representation by hard Lefschetz.  The symplectic hard Lefschetz theorem gives intersection cohomology the structure of a representation of a larger Lie algebra with more symmetries. 
    
    \begin{proposition} \label{so(4)-action}
    If\, $X$ is a primitive symplectic variety with isolated singularities, then $\IH^*(X, \mathbb C)$ inherits the structure of an $\mathfrak{so}(4)$-representation.  Moreover, the Hodge filtration on $\IH^*(X, \mathbb C)$ induced by $\sigma$ is completely determined by $\mathfrak{so}(4)$.
    \end{proposition}
    
    \begin{proof}
    The hard Lefschetz theory for the pair $(\sigma,\overline \sigma)$ (see Section~\ref{4.1}) induces two structures on $\IH^*(X, \mathbb C)$ via the $\mathfrak{sl}_2$-triples $\mathfrak g_\sigma = \langle L_\sigma, \Lambda_\sigma, H_\sigma\rangle$ and $\mathfrak g_{\overline \sigma} = \langle L_{\overline \sigma}, \Lambda_{\overline \sigma}, H_{\overline \sigma}\rangle$.  Proposition~\ref{symdualcommute} and Corollary~\ref{commutecorollary} imply that the operators 
$$
 L_\sigma, L_{\overline \sigma}, \Lambda_{\sigma}, \Lambda_{\overline \sigma}, H_\sigma, H_{\overline \sigma}
$$
 are linearly independent.  The Lie algebra $\mathfrak g_{\sigma,\overline \sigma}$ generated by these six operators is isomorphic to $\mathfrak{so}(4)$.  Moreover, $\mathfrak{g}_{\sigma,\overline \sigma}$ contains the Weil operator \begin{equation}\label{weil}
    C_\sigma = i(p-q)\id = i(H_\sigma - H_{\overline \sigma}) = i([L_\sigma, \Lambda_\sigma] - [L_{\overline \sigma}, \Lambda_{\overline \sigma}]),
    \end{equation}and so the Hodge filtration on $\IH^*(X, \mathbb C)$ is completely predicted by $\mathfrak g_{\sigma, \overline \sigma}$.
    \end{proof}
    
Therefore, the Hodge structure on intersection cohomology on any primitive symplectic variety is detected by the symplectic hard Lefschetz theorem, with no restriction on $b_2$.

\section{Representation Theory and Hodge Theory of the LLV Algebra} \label{7}

\subsection{Verbitsky Component of $\boldsymbol{\IH^*(X, \mathbb Q)}$} \label{5.4}

Let $X$ be a compact hyperk\"ahler manifold. The LLV algebra gives the cohomology ring $H^*(X, \mathbb Q)$ the structure of a $\mathfrak g$-representation.  This structure has been extensively studied in \cite{green2019llv} for the known cases of compact hyperk\"ahler manifolds and has been used to produce bounds on $b_2$ in low dimensions, see \cite{guan2001betti,sawon2015bound,kurnosov2015boundness}, and conjecturally in all dimensions; see \cite{kim2019conjectural}.  

We wish to extend some well-known results on the representation theory of the LLV algebra action on the intersection cohomology module $\IH^*(X, \mathbb Q)$ of a primitive symplectic variety $X$.  We believe that the $\mathfrak g$-structure on $\IH^*(X, \mathbb Q)$ can restrict both the number and types of singularities that primitive symplectic varieties can admit, which will be explored in future work. The first step is to understand the \textit{Verbitsky component}, which is the submodule of $\IH^*(X, \mathbb Q)$ generated by $\IH^2(X, \mathbb Q)$.

\begin{theorem} \label{verbitsky component}
Let $X$ be a primitive symplectic variety of dimension $2n$ with isolated singularities and $b_2 \ge 5$.  Then the submodule $\SIH^2(X, \mathbb Q) \subset \IH^*(X, \mathbb Q)$ generated by $\IH^2(X, \mathbb Q)$ is an irreducible $\mathfrak g$-module of\, $\Sym^*H^2(X, \mathbb Q)$.
\end{theorem}

\begin{proof}
The proof is nearly identical to the smooth case; see \cite[Theorem 1.7]{verbitsky1996cohomology} and also \cite[Theorem 2.15]{green2019llv}.  Consider the decomposition 
$$
\mathfrak g = \langle L_\alpha, \Lambda_\alpha|~Q_X(\alpha) \ne 0\rangle = \mathfrak g_2\oplus(\overline{\mathfrak g}\times \mathbb Q\cdot H)\oplus \mathfrak g_{-2},
$$
 which exists by restricting the decomposition (\ref{LLVdecomposition}) of a $\mathbb Q$-factorial terminalization, after possibly passing to a locally trivial deformation.  The semisimple part $\overline{\mathfrak g} \subset \mathfrak g_0$ of the degree 0 part of the LLV algebra acts on $\SIH^2(X, \mathbb Q)$ as it acts by derivations on cup product.  Clearly, the weight operators $H$ and $L_\alpha$ for $\alpha \in H^2(X, \mathbb Q)$ act on $\SIH^2(X, \mathbb C)$. To see that $\Lambda_\alpha$ acts on $\SIH^2(X, \mathbb Q)$, let $\alpha_1\cdots\alpha_k \in \SH^2(X, \mathbb Q)$ be any element, and note that 
$$
\Lambda_{\alpha}(\alpha_1\cdots\alpha_k) = [L_{\alpha_1},\Lambda_\alpha](\alpha_2\cdots\alpha_k) - L_{\alpha_1}(\Lambda_{\alpha_1}\alpha_2\cdots\alpha_k).
$$
 The result follows by induction on $k$ and the fact that $[L_{\alpha_1}, \Lambda_\alpha] \in \mathfrak g_0$.

 Now we may consider $\SIH^2(X, \mathbb Q)$ as a representation of $\overline{ \mathfrak g} \cong \mathfrak{so}(\IH^2(X, \mathbb C),Q_X)$.  Just as in \cite[Section~15]{verbitsky1996cohomology}, we see that
 \[
 \SIH^2(X, \mathbb Q)_{2k} =  \begin{cases} 
      \Sym^k\IH^2(X, \mathbb Q), & k \le n, \\
      \Sym^{2n-k}\IH^2(X, \mathbb Q), & k > n, 
   \end{cases}
 \]
 which depends only on the representation theory of $\mathfrak{so}(\IH^2(X, \mathbb C), Q_X)$.  Therefore,  
$$
\SIH^2(X, \mathbb Q) = \Sym^n\IH^2(X, \mathbb Q)\oplus\bigoplus_{k\ge 1}\left(\Sym^{n-k}\IH^2(X, \mathbb Q)\right)^{\oplus 2}.
$$
  It then follows  that $\SIH^2(X, \mathbb Q)$ extends to $\mathfrak g$ as an irreducible $\mathfrak g$-representation due to the ``branching rules'' for special orthogonal groups; see \cite[Section~B.2]{green2019llv}.
\end{proof}

We also get the following description of $\SIH^2(X, \mathbb Q)$, due to Bogomolov \cite{bogomolov1996cohomology} in the smooth case.  A similar description was given for a general primitive symplectic variety in terms of $H^2(X, \mathbb Q)$  in \cite[Proposition 5.11]{bakker2018global}.

\begin{proposition} \label{intersectionbogomolovresult}
Let $X$ be a primitive symplectic variety of dimension $2n$ with isolated singularities and $b_2 \ge 5$.  Then 
$$
\SIH^2(X, \mathbb Q) = \Sym^*\IH^2(X, \mathbb Q)/\langle \alpha^{n+1}|~Q_X(\alpha) = 0\rangle.
$$

\end{proposition}

The proof is completely algebraic, and the main input is the following standard lemma. 

\begin{lemma} \label{gradedlemma}
Let $(H,Q)$ be a complex vector space with a non-degenerate quadratic form $Q$, and let $A$ be a graded quotient of\, $\Sym^*H$ by a graded ideal $I^*$ such that
\begin{enumerate}
    \item $A^{2n} \ne 0$,
    \item $I^* \supset \langle \alpha^{n+1}|~Q_X(\alpha) = 0\rangle$.
\end{enumerate}
Then $I^* = \langle \alpha^{n+1}|~Q_X(\alpha) = 0\rangle$.
\end{lemma}

\noindent \textit{Proof of Proposition~\ref{intersectionbogomolovresult}}.  Lemma~\ref{gradedlemma} applies to $(\IH^2(X, \mathbb C), Q_X)$ and $A = \SIH^2(X, \mathbb C)$ by Theorem~\ref{verbitsky component}.  Note that $A^{2n} \supset \IH^{2n, 2n}(X)$ is non-zero.  The problem is invariant under small deformations, so we may assume that $X$ admits a $\mathbb Q$-factorial terminalization $\phi\colon Z \to X$.  There is a commutative diagram

\[\begin{tikzcd}
\Sym^*H^2(Z, \mathbb Q)\arrow{r} & \Sym^*H^2(Z, \mathbb Q)/\langle \alpha^{n+1}|~q_Z(\alpha) = 0\rangle \arrow{r}  &\SH^2(Z, \mathbb Q) \\
\Sym^*\IH^2(X, \mathbb Q) \arrow{rr} \arrow{u} &  &   \SIH^2(X, \mathbb Q) \arrow{u}\rlap{.}
\end{tikzcd}
\]

The natural map $\Sym^*H^2(Z, \mathbb Q) \to \SH^2(Z, \mathbb Q)$ factors through $\Sym^*H^2(Z, \mathbb Q)/\langle \alpha^{n+1}|~q_Z(\alpha) = 0\rangle$ by \cite[Proposition 5.11]{bakker2018global}.  Therefore, the bottom map must factor as 
$$
\Sym^*\IH^2(X, \mathbb Q) \lra \Sym^*\IH^2(X, \mathbb Q)/\langle \alpha^{n+1}|~Q_X(\alpha) = 0\rangle \lra \SIH^2(X, \mathbb Q)
$$
 by Lemma~\ref{IntBBFrest}.  If this map were not  injective, then Lemma~\ref{gradedlemma} would imply that the kernel would contain the rational cohomology class corresponding to the generator of $\IH^{2n,2n}(X) = (\sigma + \overline \sigma)^{2n}$, which does not vanish.  This finishes the proof. \qed 

\begin{corollary}
Let $X$ be a primitive symplectic variety of dimension $2n$ with isolated singularities and $b_2 \ge 5$.  For every $k \le n$, there is an injection 
$$
\Sym^k\IH^2(X, \mathbb Q) \longhookrightarrow \IH^{2k}(X, \mathbb Q).
$$

\end{corollary}

        \subsection{Kuga--Satake Construction on the Cohomology of Primitive Symplectic Varieties} \label{7.2}
        
        The Kuga--Satake construction, see \cite{kuga1967abelian}, associates to a polarized Hodge structure $H = H_{\mathbb Z}$ of K3 type a complex torus $T$ for which $H$ is a sub-Hodge structure of $\Hom(H_1(T), H_1(T))(1)$.  The construction associates to the weight 2 Hodge structure $H $ (with its bilinear form) its \textit{Clifford algebra} $C(H)$.  There is an induced complex structure on $C(H)\otimes \mathbb R$, and one can show that the quotient $C(H)\otimes \mathbb R/C(H)$ is a complex torus, which is in fact an abelian variety in the case that $H$ is the second cohomology of a projective K3 surface.
        
        Understanding the geometric connection between varieties admitting Hodge structures of K3 type and the Kuga--Satake construction is a difficult problem, as the Mumford--Tate group of 
$$
\Hom(H_1(T), H_1(T))(1) \cong H^1(T) \otimes H^1(T)
$$
 is highly restricted, while the Mumford--Tate group associated to a Hodge structure of K3 type can be quite large (see Section~\ref{7.3}).  In special cases, it has been observed that the Kuga--Satake construction for K3 surfaces is related to the Hodge conjecture; see \cite{geemen2000kuga}.  It is therefore interesting to understand the geometry of the Kuga--Satake construction, as well as generalizations to Hodge structures of higher weights.
        
        In \cite{kurnosov2019kuga}, Kurnosov--Soldatenkov--Verbitsky observe that there is a multidimensional Kuga--Satake construction on the cohomology ring of a compact hyperk\"ahler manifold $X$.  Namely, associated to $X$ are a complex torus $T$, a non-negative integer $l$, and embeddings 
$$
\mathfrak g \longhookrightarrow \mathfrak g_{\mathrm{tot}}(T), \quad  \Psi\colon H^*(X, \mathbb C) \longhookrightarrow H^{* + l}(T, \mathbb C),
$$
 where $\mathfrak g_{\mathrm{tot}}(T)$ is the total Lie algebra of $T$.  Here, the embedding $\Psi$ is a morphism with respect to the induced structures as $\mathfrak g$-representations (resp.\ $\mathfrak g_{\mathrm{tot}}(T)$-representations).  Fixing a complex structure on $X$ makes $\Psi$ a morphism of Hodge structures.
        
        Using the existence of the LLV algebra, we outline how the Kuga--Satake construction holds for intersection cohomology, further demonstrating how unique the geometry of primitive symplectic varieties is.
        
        \subsubsection{LLV embedding}
        
           Consider a finite-dimensional complex vector space $H$ and a non-degenerate symmetric bilinear form $Q$.  Let $T^*H$ denote the tensor algebra, and let $\mathfrak a \subset T^*H$ be the ideal generated by elements of the form 
$$
v\otimes v - Q(v,v), \quad v \in H.
$$
  The \textit{Clifford algebra} of $(H,Q)$ is $C = C(H,Q) := T^*H/\mathfrak a$.  
           
           The main technical result of \cite{kurnosov2019kuga} is the following. 
           
           \begin{theorem}[\textit{cf.} \protect{\cite[Theorem 3.14]{kurnosov2019kuga}}]\label{thm65}
           If\, $(H,Q)$ is any quadratic vector space and $W$ is any representation of\, $\mathfrak g : = \mathfrak{so}(\tilde H, \tilde Q)$, where $(\tilde H, \tilde Q)$ is the Mukai completion of\, $(H,Q)$,  then there is a $C(H,Q)$-module $V$ with an invariant symmetric bilinear form $\tau$ such that $\bigwedge^\bullet V^*$ contains $W$ as a $\mathfrak g$-submodule.
           \end{theorem} 
           
           The construction is seen by applying  Theorem~\ref{thm65} to $(\IH^2(X, \mathbb C), Q_X)$, where $X$ is a primitive symplectic variety with isolated singularities and $b_2 \ge 5$.  We let $W = \IH^*(X, \mathbb C)$, which is a $\mathfrak g$-representation by Theorem~\ref{GenLLVthm}.  Then there exists a $C = C(\IH^2(X, \mathbb C), Q_X)$-module $V$ with an embedding 
$$
\IH^2(X, \mathbb C) \longhookrightarrow \displaystyle\extp^2V^*
$$
 and an embedding of $\mathfrak g$-modules 
$$
\IH^*(X, \mathbb C) \longhookrightarrow \displaystyle\extp^\bullet V^*.
$$
 Now $\mathfrak g$ induces a grading on $\IH^*(X, \mathbb C)$ and $\displaystyle\extp^\bullet V^*$, whence we have a degree $l$ morphism \begin{equation}\label{equation embedding clifford}
              \psi\colon \IH^*(X, \mathbb C) \longhookrightarrow \displaystyle\extp^{\bullet + l}V^* 
           \end{equation} of graded vector spaces for some $l$.  If we take $T = V/\Gamma$ for some lattice $\Gamma \subset V$, we get an embedding $\Psi\colon H^*(X, \mathbb C) \hookrightarrow H^{*+l}(T, \mathbb C)$ by (\ref{equation embedding clifford}).   
           
           As we have seen, the Hodge structure on $\IH^*(X, \mathbb R)$ is detected by the Weil operators $C_\sigma \in \overline{\mathfrak g} \subset \mathfrak g$, which are determined once we fix a point $[\sigma] \in \Omega$ in the period domain.  Let $\mu \in \overline{\mathfrak g} \cong \mathfrak{so}(\IH^2(X, \mathbb R), Q_X)$ be the corresponding skew-symmetric matrix of rank 2.  Let $W_\sigma' = \langle \gamma, \gamma'\rangle$ be the positive two-space corresponding to $\sigma$, where $\gamma = \mathfrak R(\sigma), \gamma' = \mathfrak I(\sigma)$ as in Section~\ref{4.2}.  Then $\mu = \gamma\gamma' \in \overline{\mathfrak g} \subset C$.  It acts trivially on the orthogonal complement to $W_\sigma$, and $\mu^2 = -1$ in the Clifford algebra.  Thus $\mu$ defines a complex structure on $H^1(T, \mathbb R)$, noting that $T$ is \textit{smooth}.
           
           Finally, since $T$ is a complex torus, the Hodge structures on $\IH^*(X, \mathbb Q)$ and $H^*(T, \mathbb Q)$ are both determined by the symplectic hard Lefschetz theorem (see Section~\ref{3}), which is compatible with the choice $C_\sigma$.  Therefore, the morphism $\Psi$ is compatible with the Hodge structures.

        \subsubsection{Polarized Kuga--Satake construction} \label{7.2.2}
        
        Just as in the case of compact hyperk\"ahler manifolds, the existence of a polarization in (intersection) cohomology induces a polarization on the complex torus $T$ described above.
        
        Let $X$ be a projective primitive symplectic variety and $b_2 \ge 5$.  Suppose that $h \in \IH^2(X, \mathbb Q)$ is an ample class on $X$, and let $h^\perp \subset \IH^2(X, \mathbb R)$ be the orthogonal complement of $h$ with respect to the intersection BBF form $Q_X$.  The torus $T$ is a quotient of some $V$, where (as in \cite{kurnosov2019kuga}), $V = V_1^{\bigoplus N}$ with $V_1 \cong C =  C(\IH^2(X, \mathbb R), Q_X)$.
        
        Constructing the polarization on $T$ is done as follows; see \cite[Section~4.2]{kurnosov2019kuga}.  It is enough to construct a polarization on $C_h : = C(h^\perp, q|_{h^\perp})$.  Note that $(h^\perp, q_{\perp})$ is of signature $(2,k)$ for some $k$.  Let $W_h = \langle\gamma,\gamma'\rangle \subset h^\perp$ be the subspace where $q|_{h^\perp}$ is positive.  By local Torelli, we may assume that $q|_{h^\perp}(\gamma,\gamma') = 0$.  Consider the product $a = \gamma_1\gamma_2 \in C(h^\perp, q|_{h^\perp})$.  For any $x,y \in C(h^\perp, q|_{h^\perp})$, we define 
$$
\sigma_a(x,y) : = \Tr(xa\overline{y}),
$$
 where $\Tr$ is the trace map corresponding to the algebra $C$ and $\overline{y}$ is the operator 
$$
\overline{y} = \alpha\beta(y),
$$
 where $\alpha$ is the natural parity involution on the Clifford algebra and $\beta$ is the anti-automorphism which sends a tensor $y_1\otimes\cdots\otimes y_k$ to $y_k\otimes\cdots\otimes y_1$.
        
        By \cite[Proposition 4.2]{kurnosov2019kuga}, either $\sigma_a$ or $-\sigma_a$ is a polarization on $C(h^\perp, q|_{h^\perp})$.  The proof holds here, as the statement holds for any quadratic vector space $(H,q)$.
      
        \subsection{The Mumford--Tate Algebra} \label{7.3}  There is a connection between the Mumford--Tate group of the intersection cohomology of a primitive symplectic variety and the LLV algebra.  In the compact hyperk\"ahler case, this was studied in \cite[Section~2]{green2019llv}.  Given the LLV structure theorem, similar results follow through with only minor adjustments. For convenience, we reference \textit{loc.\ cit.}~to indicate the corresponding statement in the hyperk\"ahler setting.

\begin{definition}
Let $V$ be a $\mathbb Q$-Hodge structure.  The \textit{special Mumford--Tate algebra} $\overline{\mathfrak{mt}}(V)$ of $V$ is the smallest $\mathbb Q$-algebraic Lie subalgebra of $\mathfrak{gl}(V)$ such that $\overline{\mathfrak{mt}}(V)_{\mathbb R}$ contains the Weil operator $C = i(p-q)\id$.  The \textit{Mumford--Tate algebra} is defined as 
$$
\mathfrak{mt}_0(V) = \overline{\mathfrak{mt}}(V)\oplus \mathbb Q\cdot H,
$$
 where $H$ is the weight operator.
\end{definition}

\begin{proposition} \label{gen MT}
Let $X$ be a primitive symplectic variety with at worst $\mathbb Q$-factorial isolated singularities and $b_2 \ge 5$.  Let $\overline{\mathfrak m} = \overline{\mathfrak{mt}}(\IH^*(X, \mathbb Q))$ be the special Mumford--Tate algebra of the pure Hodge structure $\IH^*(X, \mathbb Q)$.  Then $\overline{\mathfrak m} \subset \overline{\mathfrak g}$, with equality if\, $X$ is very general.
\end{proposition} 

\begin{proof}
The proof is similar to that of \cite[Proposition 2.38]{green2019llv}, although the main input is that the Weil operator $C_\sigma$ with respect to a (fixed) Hodge structure on $\IH^*(X, \mathbb Q)$ is contained in the semisimple part $\overline{\mathfrak g}$ of~$\mathfrak g_0$.  The proof of \textit{loc.\ cit.}~uses the hyperk\"ahler structure, so we indicate how this works algebraically.

Recall that if we fix the point $[\sigma] \in \Omega$ in the period domain, then there is a pair $(\gamma, \gamma')$ of non-isotropic classes in $H^2(X, \mathbb R)$ such that $[L_\gamma, \Lambda_{\gamma'}] = C_\sigma$; see Corollary~\ref{intersectionVerbitskyresult}.  We also saw that by the surjectivity of the period map with respect to $H^2$,  this pair completes to a positive three-space  
$$
W_\sigma = \langle \alpha, \gamma, \gamma'\rangle \cong \mathfrak{so}(4,1).
$$
  It follows that the operators satisfy $[L_\alpha, \Lambda_\gamma] = [L_\alpha, \Lambda_{\gamma '}] = C_\sigma$ as well.  Using the commutativity of the dual Lefschetz operators, one can show that 
$$
C_\sigma = [L_\gamma, \Lambda_{\gamma'}] = -\frac{1}{2}\left[[L_\alpha, \Lambda_\gamma],[L_\alpha, \Lambda_{\gamma '}]\right] \in \overline{\mathfrak g}; 
$$
 see \cite[Proposition 2.24]{green2019llv} for the computation.  This shows that $\overline{\mathfrak{mt}}(\IH^*(X, \mathbb Q)) \subset \overline{\mathfrak g}_{\mathbb Q}$ since $\overline{\mathfrak{mt}}(\IH^*(X, \mathbb Q))_{\mathbb R}$ is the smallest subalgebra to contain the Weil operator $C_\sigma$.

For the statement regarding a general primitive symplectic variety, the proof follows as in \cite[Proposition 2.38]{green2019llv}, and we sketch the main details.  The key observation is to notice that the special Mumford--Tate group of any $\IH^k(X, \mathbb Q)$ is 
\begin{equation}\label{general MTA}
    \overline{\mathfrak{mt}}(\IH^k(X, \mathbb Q)) = \overline{\mathfrak m}. 
\end{equation}
This follows as the $\overline{\mathfrak g}$-module structure on $\IH^k(X, \mathbb Q)$ is determined by the composition 
$$
\rho_k\colon\overline{\mathfrak g} \subset \mathfrak{gl}(\IH^*(X, \mathbb Q)) \lra \mathfrak{gl}(\IH^k(X, \mathbb Q)).
$$
  In the smooth case, this map is shown to be injective; see \cite[Corollary 2.36]{green2019llv}.  This follows in the singular case, however, since the proof only depends on the representation theory of the Verbitsky component $V_{(n)} : = \SH^2(X, \mathbf C)$ (this is \cite[Proposition 2.35]{green2019llv}), which is identical to the smooth case by Theorem~\ref{verbitsky component} and Proposition~\ref{intersectionbogomolovresult}.  It follows that $\overline{\mathfrak g}$ and $\rho_k(\overline{\mathfrak g})$ are isomorphic.  Since $\overline{\mathfrak{mt}}(\IH^k(X, \mathbb Q))$ is the smallest $\mathbb Q$-algebraic subgroup such that $\rho_k(C_\sigma) \in \overline{\mathfrak{mt}}(\IH^k(X, \mathbb Q))$, we see that (\ref{general MTA}) holds. 

By the Noether--Lefschetz theory of period domains of Hodge structures of hyperk\"ahler type (see \cite{green2012mumford}), it follows that a very general primitive symplectic variety with $b_2 \ge 5$ must satisfy 
$$
\overline{\mathfrak{mt}}(H^2(X, \mathbb Q)) \cong \mathfrak{so}(H^2(X, \mathbb Q), q_X),
$$
 noting that $H^2(X, \mathbb Q)$ satisfies the local Torelli theorem; see \cite[Proposition 5.5]{bakker2018global}.
\end{proof} 

\section{Weak \texorpdfstring{$\boldsymbol{P = W}$}{P=W} for Primitive Symplectic Varieties} \label{8}

    One of the more interesting applications of the LLV algebra for compact hyperk\"ahler manifolds involves the $P = W$ conjecture.  Given a degeneration $\mathscr X \to \Delta$ of a compact hyperk\"ahler manifold $X$, the cohomology groups $H^k(X, \mathbb C)$ inherit a weight filtration from the limit mixed Hodge structure on the unique singular fiber.  There is an induced (logarithmic) monodromy operator $N \in \mathfrak{so}(H^2(X, \mathbb Q), q) \cong \overline{\mathfrak g}$, which is nilpotent of index either 2 or 3.  We say a degeneration is of type III if $N$ has index 3.
    
    Any compact hyperk\"ahler manifold admits a type III degeneration; see \cite{soldatenkov2018limit}.  The $P = W$ conjecture for Lagrangian fibrations states that the induced weight filtration from the limit mixed Hodge structure agrees with the perverse filtration when $X$ admits a Lagrangian fibration.  This was answered positively in \cite{harder2021p} by showing that the data of these filtrations agree with the Hodge filtration induced from a positive three-space $W_g$ corresponding to a hyperk\"ahler metric $g$.
    
    We can form an analog of the $P = W$ conjecture for primitive symplectic varieties admitting a Lagrangian fibration, relating the data of a Lagrangian fibration to the filtration induced by the logarithmic monodromy operator $N$ of a type III degeneration.  We expect that the filtration induced by $N$ agrees with a limit mixed Hodge structure for intersection cohomology, although there is a subtle issue describing the intersection cohomology in terms of a variation of pure Hodge structures rather than the underlying pure Hodge module.  This will be explored in future work.
    
    \subsection{Perverse $\boldsymbol{=}$ Hodge}

    The LLV algebra for a compact hyperk\"ahler manifold detects the information of a Lagrangian fibration.  We outline how some of these results hold in the case of primitive symplectic varieties, which is based purely on the work of Shen--Yin \cite{shen2022topology}; see also \cite{huybrechts2022lagrangian} for a survey and \cite{felisetti2022intersection} in the case of primitive symplectic varieties admitting a symplectic resolution.

    \subsubsection{Perverse filtration on the cohomology of a Lagrangian fibration}

    Throughout this section, we assume that $X$ is a primitive symplectic variety of dimension $2n$, admitting a Lagrangian fibration $f\colon X \to B$ to a projective base $B$ of dimension $n$. 
    
    Given a projective morphism $f\colon X \to B$, there is a natural filtration on the cohomology of $X$ induced from the images of the truncated complexes of the perverse $t$-structure associated to the morphism $f$: 
$$
P_mH^k(X, \mathbb C) = \im\left(\mathbb H^{k-2n}(B, {}^{\mathfrak p}\tau_{\le m}(\mathbf Rf_*(\mathcal{IC}_X\otimes \mathbb C)[-2n])) \lra H^k(X, \mathbb C)\right).
$$
 The filtration is completely determined by an ample class on the base $B$.  Indeed, if $\alpha \in H^2(B, \mathbb R)$ is ample and $\beta = f^*\alpha$, then 
$$
P_mH^k(X, \mathbb C) = \sum_i\left(\ker\left(L_\beta^{2n + m + i - k}\right)\cap \im\left(L_\beta^{i-1}\right)\right)\cap \IH^k(X, \mathbb C), 
$$
 where $L_\beta$ is the cupping operator, see \cite[Proposition 5.2.4]{de2005hodge}.
    
    By the Fujiki relations on the BBF form $q_X$, the pullback $\beta$ is $q_X$-isotropic and therefore $Q_X$-isotropic.  For any isotropic class $\mu \in \IH^2(X, \mathbb C)$, we may define an analogous filtration 
$$
P_m^\mu H^k(X, \mathbb C) = \sum_i\left(\ker\left(L_\mu^{2n + m + i - k}\right)\cap \im\left(L_\mu^{i-1}\right)\right)\cap \IH^k(X, \mathbb C).
$$

    \begin{lemma} \label{lemma invariance perverse dimension}
    For any $Q_X$-isotropic classes $\mu_1,\mu_2$, we have 
$$
\dim P_m^{\mu_1}\IH^k(X, \mathbb C) = \dim P_m^{\mu_2}\IH^k(X, \mathbb C).
$$

    \end{lemma}
    
    \begin{proof}
    For classes $\mu_1,\mu_2 \in H^2(X, \mathbb C)$, this is \cite[Proposition 1.2]{felisetti2022intersection}, and so the lemma holds in the $\mathbb Q$-factorial case.  By taking a $\mathbb Q$-factorial terminalization $\phi\colon Z \to X$ (noting that $X$ is projective), we see that $\dim P_m^{\phi^*\mu_1}H^k(Z, \mathbb C) = \dim P_m^{\phi^*\mu_2}H^k(Z, \mathbb C)$.  But this implies the statement of the theorem as $\phi\colon Z \to X$ is semismall.  Indeed, we must have that ${}^{\mathfrak p}\mathscr H^j(\phi_*\mathcal{IC}_Z[2n]) = 0$ for every $j \ne 0$, so the induced perverse filtration with respect to the morphism $\phi\colon Z \to X$ is trivial, and the dimension of the perverse filtration will be determined by the pullback.
    \end{proof}
    
    Now consider the filtration $P_m^{\overline \sigma}\IH^k(X, \mathbb C)$ associated to the antiholomorphic symplectic form $\overline \sigma$.  It is an increasing filtration which detects the Hodge filtration by Theorem~\ref{symplecticsymmetry}.  Specifically, 
$$
P_m^{\overline \sigma}\IH^k(X, \mathbb C) = \bigoplus_{p \le m} \IH^{p, m-p}(X).
$$
  We therefore see that the Hodge numbers equal the perverse Hodge numbers: 
$$
\dim \IH^{p,q}(X) = {}^{\mathfrak p}h^{p,q} : = \dim \gr_p^P \IH^m(X, \mathbb C).
$$

    \subsubsection{A Lefschetz class corresponding to $\boldsymbol{(\beta, \eta)}$} \label{subsubsection data lagrangian}  As in the smooth case, the data of a $\mathbb Q$-factorial terminal primitive symplectic variety admitting Lagrangian fibration is encoded in the LLV algebra; see \cite{shen2022topology,harder2021p}.  Let $\beta$ be as above (which is $q_X$-isotropic by the Fujiki relations), and let $\eta \in H^2(X, \mathbb Q)$ be an $f$-relative ample class.  By replacing $\eta$ with a $\mathbb Q$-linear combination of $\eta$ and $\beta$ as needed, we may assume $q_X(\eta) = 0$.  By global Torelli, there is a class $\rho \in H^2(X, \mathbb Q)$ such that 
$$
q_X(\rho) > 0, \quad q_X(\eta, \rho) = q_X(\beta, \rho) = 0.
$$
  The corresponding Lie algebra $\mathfrak g_{\rho}$ generated by the simultaneous $\mathfrak{sl}_2$-triples induced by $\rho, \beta, \eta$ naturally sits inside $\overline{\mathfrak g}\cong \mathfrak{so}(H^2(X, \mathbb C), q_X)$; see \cite[Equation (5)]{harder2021p}.  The Lie algebra $\mathfrak g_\rho$ inherits $\IH^*(X, \mathbb C)$ with the structure of an $\mathfrak{so}(5)$-representation.

    By \cite[Proposition 1.1]{shen2022topology}, there is a canonical splitting of the perverse filtration: \begin{equation}\label{equation perverse splitting} P_l^\beta \IH^k(X, \mathbb Q) = \bigoplus_{p + q = l} P^{p,q}.\end{equation} We note that the proof is stated for compact hyperk\"ahler manifolds, but it is completely algebraic.  The contents of Sections~\ref{4} and~\ref{5} immediately imply (\ref{equation perverse splitting}).

    \subsection{Degenerations}\label{section type III} Soldatenkov's proof of the existence of maximally unipotent degenerations is based purely on lattice theory and knowledge of the period domain of compact hyperk\"ahler manifolds.  As primitive symplectic varieties satisfy global Torelli, the existence of degenerations for locally trivial families will follow exactly as in the smooth case.
    
    \begin{definition}
    Let $X$ be a primitive symplectic variety.  A \textit{degeneration} of $X$ is a flat proper morphism $g\colon\mathscr X \to \Delta$ of complex analytic spaces such that 
    \begin{enumerate}
        \item for some $t \in \Delta^*$, the fiber satisfies $\mathscr X_t \cong X$; 
        \item the restriction $g'\colon \mathscr X^* \to \Delta^*$ is a locally trivial deformation; and
        \item the monodromy action on the $H^2(\mathscr X_t, \mathbb Q)$ is unipotent and non-trivial.
    \end{enumerate}
    We say that a degeneration is \textit{projective} if $g$ is a projective morphism.
    \end{definition}

    If $g$ is a degeneration of $X$, then every fiber $\mathscr X_t$ is also primitive symplectic, and each fiber is $\mathbb Q$-factorial terminal if $X$ is. Moreover, $R^2g'_*\mathbb Z$, where $g'\colon\mathscr X^* \to \Delta^*$ is the restriction, is a local system as we restrict ourselves to locally trivial deformations.  In particular, we get a variation of pure Hodge structures $\mathscr V$, where each fiber is isomorphic to $\Gamma_{\mathbb Z} \cong H^2(\mathscr X_t, \mathbb Z) \cong H^2(X, \mathbb Z)$.  If $h \in \Gamma_{\mathbb Z}$ is the class of a polarization, let $\mathscr V^h$ be the $q_X$-orthogonal complement of the bilinear pairing induced by the Beauville--Bogomolov--Fujiki form on~$\mathscr V$.  It forms a local system with fiber equal to the $q$-complement of $h$, that is, $\Gamma_{\mathbb Z}^h$.

    \subsubsection{Limit mixed Hodge structure for intersection cohomology} \label{subsubsection lmhs} The weight filtration induced by a degeneration of primitive symplectic varieties is a consequence of Schmid's work on the \textit{limit mixed Hodge structure}; see \cite{schmid1973variation}.  Let $\mathscr V$ be an (integral) variation of pure Hodge structures over $\Delta^*$ admitting a maximally unipotent monodromy operator $T$, and let $V_0 = \overline{\mathscr V}_0$ be the fiber of the unique extension of $\mathscr V$ over $\Delta$.  For each $k \in \mathbb Z_{ \ge 0}$, the corresponding log-monodromy operator $N$ defines a unique increasing weight filtration $W^\bullet_k$ satisfying the following properties: \begin{enumerate}
        \item $NW^j_k \subset W^{j-2}_k$ for $j \ge 2$, and
        \item the induced map $N^l\colon \gr_{j + l}^{W_k}V_0 \to \gr_{j-l}^{W_k}V_0$ is an isomorphism for all $l \ge 0$; 
    \end{enumerate} see \cite[Lemma 6.4]{schmid1973variation}. In particular, the triple $(V_0, W_k^\bullet, \overline{\mathcal F}^\bullet_0)$ gives rise to a mixed Hodge structure, where $\overline{\mathcal F}^\bullet$ is a holomorphic extension of the Hodge bundle $\mathcal F^\bullet$ underlying $\mathscr V$. 

    The main application of the nilpotent weight filtration is on the cohomology of $X$ induced by a degeneration of compact K\"ahler manifolds.  Extending Proposition~\ref{intloc}, we can describe a limit mixed Hodge structure on the intersection cohomology of the central fiber of a degeneration $g\colon \mathscr X \to \Delta$ of a primitive symplectic variety $X$.  By definition, the locally trivial family $g'\colon \mathscr X^* \to \Delta^*$ admits a simultaneous resolution of singularities $f'\colon \mathscr Y^* \to \Delta^*$; see \cite[Lemma 4.9]{bakker2018global}. For each $k$, let $\mathscr{H}_{\mathscr Y}^k$ be the variation of pure Hodge structures with local system $R^kf'_*\mathbb Q_{\mathscr Y}$.  If $\mathbb{IH}^k$ is the local system determined by the intersection cohomology of the fibers $(g')^{-1}(t)$, the decomposition theorem (see Proposition~\ref{intprops}) implies that $\mathbb{IH}^k$ underlies a sub-variation of pure Hodge structures $\mathscr H_{\mathscr X}^k$ of $\mathscr H_{\mathscr Y}^k$.  In particular, we have the following. 

    \begin{defthm}
        If $g\colon \mathscr X \to \Delta$ is a degeneration of a primitive symplectic variety $X$, there is a mixed Hodge structure, called the \textit{limit mixed Hodge structure}, on the intersection cohomology $\IH^k(X_\infty, \mathbb Q)$ of the canonical fiber $X_\infty$.
    \end{defthm}

    \subsubsection{Existence of type III degenerations}
    
    There is an induced monodromy transformation $\lambda \in \Aut(\mathscr V^h, q) \cong O(\Gamma_{\mathbb Z}^h, q)$ which, by definition, must be of the form $\lambda = e^N$, where $N \in \mathfrak{so}(\Gamma^h_{\mathbb Q},q)$, and, by \cite[Theorem 6.1]{schmid1973variation}, must be of index 2 or 3.  We say that a degeneration is \textit{maximally unipotent}, or has \textit{maximally unipotent monodromy}, if $N$ is of index 3.
    
    \begin{proposition} \label{type III existence}
    Let $X$ be a primitive symplectic variety with at worst $\mathbb Q$-factorial isolated singularities and $b_2 \ge 5$. There exists a projective degeneration of $X$ with maximally unipotent monodromy.
    \end{proposition}
    
    \begin{proof}
    The proof follows as in \cite[Section~4]{soldatenkov2018limit} almost verbatim, and so we  only briefly indicate the main details.  Since $b_2 \ge 5$ and $(\Gamma_{\mathbb Z}, q)$ is of signature $(3,b_2-3)$ as in the smooth case, there exist a polarization $h \in \Gamma_{\mathbb Z}$ and an endomorphism $N \in \mathfrak{so}(\Gamma_{\mathbb Q}^h, q)$ of index 3; see \cite[Lemma 4.1]{soldatenkov2018limit}.  Moreover, the restriction of~$q$ to the image of $N$ is semi-positive with one-dimensional kernel.
    
    Let $\Omega$ be the period domain with respect to $(\Gamma_{\mathbb Z}, q)$, and let $\widehat{\Omega}$ be the compact dual.  The polarization defines a period domain $\Omega^h$ with compact dual $\widehat{\Omega}^h$.  For $N \in \mathfrak{so}(\Gamma_{\mathbb Q}^h, q)$ and $x \in \widehat{\Omega}^h$, Soldatenkov defines the pair $(N,x)$ to be a \textit{nilpotent orbit} if $e^{itN} \in \Omega^h$ for every $t >> 0$.  Equivalently, see \cite[Lemma 4.4]{soldatenkov2018limit}, $(N,x)$ is a nilpotent orbit if and only if $q(Nx, N\overline{x}) > 0$, and the image of such points in $\widehat{\Omega}^h$ therefore defines a non-empty open subset.  For such a nilpotent monodromy operator $N$ of index 3, let $\mathscr N = \{x \in \Omega^h\mid(N,x)~\mathrm{is~a~nilpotent~orbit}\}$.  It is open and non-empty, whence we get an open subset of the period domain which corresponds to this nilpotent operator $N$.  As in \cite[Theorem 4.6]{soldatenkov2018limit}, this open subset predicts a degeneration $\mathscr X \to \Delta$ of $X$ with logarithmic monodromy operator $N$.
    \end{proof}   
    
    \subsection{Singular $\boldsymbol{P = W}$ Theorem} We can now state a singular version of the Lagrangian $P = W$ conjecture, as follows. 
    
    \begin{theorem}
      Let $X$ be a $\mathbb Q$-factorial terminal primitive symplectic variety with isolated singularities
      and $b_2 \ge 5$.  If $f\colon X \to B$ is a Lagrangian fibration, the perverse filtration $P^\beta$ on $\IH^*(X, \mathbb C)$ with respect to the pullback of an ample class on $B$ agrees with the weight filtration $W_N$ on $\IH^*(X_\infty, \mathbb C)$ with respect to the logarithmic monodromy operator $N$ of a type III degeneration $\mathscr X \to \Delta$ of\, $X$.
    \end{theorem}
    
    \begin{proof}
The proof follows \cite[Sections~9--11]{harder2021p}. 
 Note that by Proposition~\ref{Qcrit} and Lemma~\ref{IntBBFrest},   
$$
(H^2(X, \mathbb Q), q_X) = (\IH^2(X, \mathbb Q), Q_X).
$$
  Let $\rho, \beta, \eta$ be the triple associated to the Lagrangian fibration $f\colon X \to B$ (see Section~\ref{subsubsection data lagrangian}), and let $\mathfrak g_{\rho} \cong \mathfrak{so}(5)$ be the Lie algebra associated to $\rho$.  By \cite[Lemma 4.1]{soldatenkov2018limit}, there is a nilpotent operator\footnote{The proof uses Meyer's theorem on the lattice $(H^2(X, \mathbb Z), q_X)$, which also requires $b_2 \ge 5$.} $N_{\beta,\rho}$ of index 3 in $\mathfrak{so}(H^2(X, \mathbb Q), q_X)$ corresponding to the pair $(\beta, \rho)$.  By (\ref{equation semisimple isomorphism}) and \cite[Lemma 3.9]{kurnosov2019kuga}, we have $N_{\beta, \rho} = [L_{\beta}, \Lambda_{\rho}] \in \mathfrak g_{\rho} \subset \overline{\mathfrak g}$.

 Let $W_N^\bullet$ be the weight filtration corresponding to the completion of the nilpotent operator $N_{\beta, \rho}$ to an $\mathfrak{sl}_2$-triple.  By Proposition~\ref{type III existence} and Section~\ref{subsubsection lmhs}, $W_N^\bullet$ restricts to the weight filtration of the limit mixed Hodge structure $\IH^k(X, \mathbb Q)$ of the corresponding degeneration $\mathscr X \to \Delta$ with logarithmic monodromy operator $N_{\rho, \beta}$.  On the other hand, by local Torelli and Section~\ref{4}, the isotropic pairs $(\beta, \eta)$ induce an $\mathfrak{sl}_2\times \mathfrak{sl}_2$-action with weight decomposition
$$
\IH^*(X, \mathbb C) = \bigoplus_{p,q} H^{p,q}
$$
 such that the corresponding weight operators $H_\beta, H_\eta$ satisfy 
$$
H_{\beta}|_{H^{p,q}} = (q-n)\id, \quad H_{\eta}|_{H^{p,q}} = (p-n)\id.
$$

 From the proof of Lemma~\ref{lemma invariance perverse dimension}, local Torelli, and the symplectic hard Lefschetz theory of Section~\ref{4.1}, it follows that $H^{p,q} = P^{p,q}$, where the $P^{p,q}$ are the summands of the splitting (\ref{equation perverse splitting}).  It follows that the perverse filtration restricted to $\IH^k(X, \mathbb C)$ agrees with the weight filtration on $\IH^k(X_\infty, \mathbb C)$.

    \end{proof}


\end{document}